\documentclass[final,leqno]{siamltex}
\usepackage{amsmath,amssymb}
\usepackage{graphicx,color}
\usepackage{enumerate}
\usepackage{verbatim}

\newtheorem{remark}{Remark}[section]

\newcommand{\csta}{C_{\rm sta}}
\newcommand{\hatcsta}{\widehat{C}_{\rm sta}}
\newcommand{\om}{\omega}
\newcommand{\Om}{\Omega}
\newcommand{\pa}{\partial}
\renewcommand{\i}{{\rm\mathbf i}}

\DeclareMathOperator{\re}{{Re}}
\DeclareMathOperator{\im}{{Im}}
\newcommand{\ddiv}{\mbox{\rm div\,}}
\newcommand{\curl}{\mbox{\rm {\bf curl\,}}}
\newcommand{\bV}{\mathbf{V}}
\newcommand{\bH}{\mathbf{H}}
\newcommand{\bff}{\mathbf{f}}

\newcommand{\bfu}{\mathbf{u}}
\newcommand{\bfv}{\mathbf{v}}

\newcommand{\bfx}{\mathbf{x}}

\newcommand{\bfE}{\mathbf{E}}
\newcommand{\bftE}{\widetilde{\mathbf{E}}}

\newcommand{\bfF}{\mathbf{F}}

\newcommand{\bfP}{\mathbf{P}}

\newcommand{\bfS}{\mathbf{S}}

\newcommand{\bfV}{\mathbf{V}}

\newcommand{\bfnu}{\boldsymbol{\nu}}
\newcommand{\bfpsi}{\boldsymbol{\psi}}
\newcommand{\bfPsi}{\boldsymbol{\Psi}}
\newcommand{\bftPsi}{\widetilde{\bfPsi}}
\newcommand{\bfphi}{\boldsymbol{\phi}}

\newcommand{\bcV}{\boldsymbol{\mathcal{V}}}

\newcommand{\ba}{\mathbf{a}}
\newcommand{\bfb}{\mathbf{b}}
\newcommand{\bfc}{\mathbf{c}}
\newcommand{\bfd}{\mathbf{d}}

\newcommand{\ome}{\omega}
\newcommand{\Ome}{\Omega}
\newcommand{\bL}{\mathbf{L}}

\renewcommand{\figurename}{\scriptsize Figure}

\newcommand{\E}{\mathbb{E}}

\newcommand{\Div}{{\rm div}}
\newcommand{\veps}{\varepsilon}

\renewcommand{\i}{{\rm\mathbf i}}

\newcommand{\bbf}{\mathbf{f}}

\newcommand{\vep}{\varepsilon}

\def\be{\begin{equation}}
\def\ee{\end{equation}}
\def\br{\begin{eqnarray}}
\def\er{\end{eqnarray}}

\title{An efficient Monte Carlo interior penalty discontinuous Galerkin method for the time-harmonic 
	Maxwell's Equations with random coefficients}
\author{
Xiaobing Feng\thanks{Department of Mathematics, The University of
Tennessee, Knoxville, TN 37996, U.S.A.  ({\tt xfeng@math.utk.edu}). The work of this author was 
partially supported by the NSF grant DMS-1620168.}
\and
Junshan Lin\thanks{Department of Mathematics and Statistics, Auburn University,
Auburn, AL 36849, U.S.A. ({\tt jzl0097@auburn.edu}) The work of this author was 
partially supported by the NSF grant DMS-1719851.}
\and
Cody Lorton\thanks{Department of Mathematics and Statistics, University of
West Florida, Pensacola, FL 32514, U.S.A.  ({\tt clorton@uwf.edu}).}
}

\begin{document}

\maketitle

\begin{abstract}
This paper develops an efficient Monte Carlo interior penalty discontinuous 
Galerkin method for electromagnetic wave propagation in random media. 
This method is based on a multi-modes expansion of the solution 
to the time-harmonic random Maxwell equations. It is shown
that each mode function satisfies a Maxwell system with random sources defined recursively.
 An unconditionally stable IP-DG method is employed to discretize the 
nearly deterministic Maxwell system and the Monte Carlo method combined with 
an efficient acceleration strategy is proposed for computing the mode 
functions and the statistics of the electromagnetic wave.
A complete error analysis is established for the 
proposed multi-modes Monte Carlo IP-DG method. It is proved that the proposed 
method converges with an optimal order for each of three levels of 
approximations. Numerical experiments are provided to validate the theoretical 
results and to gauge the performance of the proposed numerical  method and approach. 
\end{abstract}


\section{Introduction}

The study of electromagnetic wave propagation in random media, such as 
atmosphere and biological media, 
has been a subject of interest for decades due to its  in applications 
in communication, remote sensing, detection, imaging, etc 
\cite{Ishimaru_78, Andrews_Phillips_05, Tsang_Kong_Shin_85}.
In such instances, it is of practical interest to characterize the statistics 
of the electromagnetic wave field scattered by the random media.
However, even with the rapid development of modern computing power, numerical 
modeling of the full three-dimensional Maxwell's equations with random coefficients
is still a challenging task. This not only has to do with the large scale of the problem
and its uncertainty 
but also is related to the modeling of multiple scattering effects for wave propagation in random media. 
Typically, existing methods such as 
direct Monte Carlo techniques for sampling the random media and the corresponding solution or stochastic Galerkin methods by representing the random solution with the
Karhunen--Lo\`{e}ve or Wiener Chaos expansion are still computationally intractable for solving random vector Maxwell's equations in three-dimensions \cite{FGPS,Lord_Powell_Shardlow}.

In this paper, we present an efficient Monte-Carlo interior penalty discontinuous Galerkin (MCIP-DG) method
for the characterization of the statistics of an electromagnetic wave in random media.
Let $D$ be a convex polygonal domain in $\mathbb{R}^3$, and $(\Omega, \mathcal{F}, P)$ be the probability space with the sample space $\Omega$, the $\sigma$-algebra $\mathcal{F}$, and the probability measure $P$. 
We consider the following time-harmonic Maxwell problem for the electric field $\bfE$:  
\begin{alignat}{2}
	\curl \curl \bfE(\om,\cdot) - k^2 \alpha(\om,\cdot)^2 \bfE(\om,
	\cdot) &= \bff(\om,\cdot) &&\quad  \mbox{in } D, \label{Eq:PDE1}\\
	 \curl \bfE(\om,\cdot) \times \bfnu - \i k \lambda \bfE_T(\om,\cdot) &= 
	 \mathbf{0} &&\quad  \mbox{on } \pa D,
	 \label{Eq:PDE2}
\end{alignat}
for almost all random samples $\om\in\Omega$.  Here $k > 0$ is the wave number, $
\lambda > 0$ is an impedance parameter, and $\bfnu
$ is the outward normal to the boundary $\pa D$. In addition, we use $\bfE_T$ to denote
the tangential projection of $\bfE$ on $\pa D$, which is given by $$\bfE_T := (\bfnu \times  \bfE) \times \bfnu.$$ 
The boundary condition \eqref{Eq:PDE2} is called impedance boundary condition in electromagnetism (c.f. \cite{Colton_Kress_1992}). 

The index of refraction, $\alpha(\omega,\bfx)$, is a random field such that for each fixed point 
$\bfx\in D$, $\alpha(\cdot,\bfx)$ is a random variable.  In this paper, we consider weakly random media 
in the sense that $\alpha$  has  the following form:
\begin{align}\label{Eq:AlphaDef}
	\alpha(\om, \cdot) := 1 + \vep \eta(\om,\cdot).
\end{align}
This means that $\alpha$ is a small random perturbation of a deterministic background medium.
Here, $\vep > 0$ denotes the perturbation parameter and the random field $\eta \in 
L^2(\Om,W_C^{1,\infty}(D))$ satisfies
\begin{align*}
	P \left\{ \om \in \Om; \| \eta(\om,\cdot) \|_{L^\infty(D)} \leq 
	1 \right\} &= 1, \\
	P \left\{ \om \in \Om; \| \nabla \eta(\om,\cdot) \|_{L^\infty(D) 
	}\leq \mu \right\} &= 1,
\end{align*}
where $\mu> 0$ is a given constant. At the end of the paper, we shall also present a procedure 
for dealing with more general random field $\alpha$. 

The numerical method presented here is based on a multi-modes representation of the electric field $\bfE$. 
The expansion yields the same deterministic Maxwell's equation with recursively defined random sources 
for all modes, which is the key for us to design an efficient MCIP-DG method and speed up the whole computational algorithm.
In the algorithm we employ an absolutely stable  MCIP-DG method to
approximate each mode function which satisfies a nearly deterministic Maxwell's system.
The acceleration of the numerical method is achieved by performing an LU decomposition 
of the IP-DG stiffness matrix for the Maxwell operator, and
all samples at every order can be obtained in an efficient manner by simple forward and backward substitutions. 
This significantly reduces the computational cost for computing the electric field $\bfE$ for each sample.
The proposed numerical method nontrivally extends our previous studies in \cite{Feng_Lin_Lorton_16, Feng_Lin_Nicholls_18, Feng_Lorton_17} to the full vector Maxwell's equations in three dimensions. 
For the Maxwell's equations, the wave-number-explicit estimation for the solution 
$\bfE$ involves extra complications arising from the estimation of $\ddiv \bfE$.
This gives rises to additional difficulties for the analysis compared to our previous 
works for the random scalar or elastic Helmholtz equations. It also imposes new constraints 
on the random media to ensure the convergence (see Section \ref{sec:PDE_Analysis} for details).

The rest of the paper is organized as follows. We derive wave-number-explicit estimates for the
solution of the random Maxwell's equations in Section \ref{sec:PDE_Analysis}. This analysis lays the foundation for 
the convergence analysis of the multi-modes expansion and the numerical analysis for the overall numerical 
algorithm. In Section \ref{sec:Multi-Modes}, we introduce the multi-modes expansion of the electric field as 
a power series of $\veps$ and establish the error estimation for its finite-modes approximation.
The Monte Carlo interior penalty discontinuous Galerkin method is presented in Section 
\ref{sec:MCIP-DG}, which is used to approximate each mode function by solving a deterministic
Maxwell's system with a random source term. In Section \ref{sec:Numerical_Procedure}, we present 
a complete numerical algorithm for solving the random Maxwell's equations
\eqref{Eq:PDE1}--\eqref{Eq:PDE2}, and derive the error estimations for the proposed algorithm.
Several numerical experiments are provided in Section \ref{sec:Numerical_Experiments} to demonstrate 
the efficiency of the method and to validate the theoretical results. We end the paper with a discussion on generalization of the proposed numerical method to more general random media in Section \ref{sec-7}.

\section{PDE analysis} \label{sec:PDE_Analysis}  
\subsection{Preliminaries}

Let $\E(\cdot)$ denote the expectation operator defined over the probability space $(\Om,\mathcal{F},P)$ and is given by
\begin{align*}
	\E (u) := \int_\Om u \, dP.
\end{align*}
Throughout the paper, we will assume the spatial domain $D \subset B_R(\bf0)$ and it is star-shaped with respect to the 
origin such that
\begin{align}\label{Eq:StarShape}
	\bfx \cdot \bfnu \geq c_0 \mbox{ on } \pa D.
\end{align}
Let $\bL^2(D) = (L^2(D))^3$ and  $\bL^2(\pa D)$ be the vector space of complex, 
vector valued square integrable functions on a domain $D$ and its boundary $\pa D$, respectively. They are equipped with the standard inner products
\begin{align*}
	(\bfu, \bfv)_D := \int_D \bfu \cdot \overline{\bfv} \, d\bfx, \qquad
	\langle \bfu, \bfv \rangle_{\pa D} := \int_{\pa D} \bfu \cdot
	\overline{\bfv} \, dS,
\end{align*}
respectively.  In addition, we define the following function spaces:
\begin{align*}
\bH(\curl, D)&:=\Big\{ \bfv\in \bL^2(D) \Big| \, \curl \bfv\in  \bL^2(D) \Big\},\\
\bH(\ddiv, D)&:=\Big\{ \bfv\in \bL^2(D) \Big| \, \ddiv \bfv\in  L^2(D) \Big\},\\
\bcV &:=\Big\{\bfv\in \bH(\curl, D) \Big| \, \bfv_T \in \bL^2(\pa D) \Big\},\\
\hat{\bcV} &:= \Big\{ \bfv\in \bH(\curl, D) \Big| \, \curl \bfv \in \bH(\curl, D) \mbox{ and } \bfv_T \in \bH(\curl, \pa D) \Big\}.
\end{align*}
The weak solution to \eqref{Eq:PDE1}--\eqref{Eq:PDE2} is defined as follows.
\begin{definition} \label{Def:WeakForm}
	Let $\bff : \Omega \to \bH(\ddiv,D)$.  A function 
	$\bfE: \Omega \to \bcV$ is called a weak 
	solution to problem \eqref{Eq:PDE1}--\eqref{Eq:PDE2} if it satisfies 
	\begin{align}
		a\big(\bfE(\om,\cdot),\bfv\big) &= (\bff(\om,\cdot), \bfv)_D 
		,  & \forall \, \bfv \in  \bcV, \label{Eq:WeakForm}
	\end{align}
	almost surely, where
	\begin{align} \label{Eq:SesqForm}
	a\big(\bfE(\om,\cdot),\bfv \big)
	&:= \big( \curl \bfE(\om,\cdot), \curl 
	\bfv \big)_D 
	- k^2 \big(\alpha(\om,\cdot) \bfE(\om,\cdot), \bfv \big)_D \\
	  &\quad- \i k \lambda 	\big \langle \bfE_T(\om,\cdot), \bfv_T \big \rangle_{\pa D}. \nonumber
\end{align}
\end{definition}

\begin{remark}\label{rem2.1}
	The pathwise formulation \eqref{Eq:WeakForm} can be replaced by an averaged 
	formulation obtained by taking expectation on both sides of \eqref{Eq:WeakForm}.  
	It can be shown that both formulations are equivalent provided that $\bff\in L^2(\Omega, \bH(\Div,D))$.  
\end{remark}

\subsection{Wavenumber-explicit solution estimates}
We start by stating some key lemmas that will be used later to establish the 
stability estimate for the solution
of \eqref{Eq:PDE1}-\eqref{Eq:PDE2}.
The following lemma gives Rellich identities for the time-harmonic Maxwell's equations.  These identities were used in \cite{Feng_Wu_14,Lorton_14} to derive solution estimates for the deterministic time-harmonic Maxwell's equations.
\begin{lemma} \label{Lem:RellichIDs}
	Let $\bfE: \Om \to \hat{\bcV}$ and
	define $\bfv := \curl \bfE \times \bfx$, then the 
	following identities hold almost surely:
	\begin{align}
		&2 \re \big(\curl \bfE(\om,\cdot), \curl \bfv(\om,\cdot) \big)_D = 
		\| \curl \bfE(\om,\cdot) \|^2_{L^2(D)} \label{Eq:Rellich1} \\
		& \qquad \qquad + \big 
		\langle \bfx \cdot \bfnu, | \curl \bfE(\om,\cdot) |^2 \big 
		\rangle_{\pa D}, \notag\\
		&2 \re \big(\bfE(\om,\cdot),\bfv(\om,\cdot) \big)_D  = - \| 
		\bfE(\om,\cdot) \|^2_{L^2(D)} -  \big \langle \bfx \cdot 
		\bfnu, |\bfE(\om,\cdot)|^2 \big \rangle_{\pa D}  
		\label{Eq:Rellich2}\\
		& \qquad \qquad + 2 \re \Big( \big(\bfx \, \ddiv 
		\bfE(\om,\cdot), \bfE(\om,\cdot) \big)_D + \big \langle \bfx \times 
		\bfE(\om,\cdot), \bfE(\om,\cdot) \times \bfnu \big \rangle_{\pa D} 
		\Big)  .\notag
	\end{align}
\end{lemma}

The next couple of estimates will be instrumental for the full solution estimates in Theorem \ref{Thm:PDEEstimate}.
\begin{lemma} \label{Lem:PDEEstimate1}
Let $\bfE: \Om \to \hat{\bcV}$ be the weak solution to \eqref{Eq:PDE1}--\eqref{Eq:PDE2}.  Then for any $\delta_1, \delta_2 > 0$ and $0 \leq \veps < 1$, the following estimates hold almost surely:
\begin{align}
	\label{Eq:PDEEstLem1} & \| \curl \bfE(\om,\cdot) \|^2_{L^2(D)} 
	\leq \big( k^2(1+\veps)^2 + \delta_1 \big) \| 
	\bfE(\om,\cdot) \|^2_{L^2(D)} +  \frac{1}{4 \delta_1} \| \bff(\om,\cdot) 
	\|^2_{L^2(D)} , \\
	\label{Eq:PDEEstLem2} & \| \bfE_T(\om,\cdot) \|^2_{L^2(\pa D)} 
	\leq \delta_2 \| \bfE(\om,\cdot) \|^2_{L^2(D)} + \frac{1}{4 k^2 \lambda^2 
	\delta_2} \| \bff(\om,\cdot) \|^2_{L^2(D)}.
\end{align}
\end{lemma}
\begin{proof}
	Setting $\bfv = \bfE$ in \eqref{Eq:WeakForm} yields
	\begin{align*}
		& \| \curl \bfE(\om,\cdot) \|^2_{L^2(D)} - k^2 \big((1+
		\veps \eta(\om,\cdot))^2, | \bfE(\om,\cdot) |^2 \big)_D - \i k \lambda 
		\| \bfE_T(\om,\cdot) \|_{L^2(D)}^2  \\
		& \qquad \qquad = \big(\bff(\om,\cdot), \bfE(\om,\cdot) \big)_D,
	\end{align*}
	almost surely.
	Taking the real and imaginary part separately gives
	\begin{align}
		\| \curl \bfE(\om,\cdot) \|^2_{L^2(D)} - k^2 
		\|(1+\veps \eta(\om,\cdot)) \bfE(\om,\cdot) \|^2_{L^2(D)} 
		 &= \re\big(\bff(\om,\cdot), \bfE(\om,\cdot) \big)_D, 
		 \label{Eq:PDEEstLem1a}\\
		k \lambda \| \bfE_T(\om,\cdot) \|_{L^2(D)}
		^2 &= - \im \big(\bff(\om,\cdot), \bfE(\om,\cdot) \big)_D, 
		\label{Eq:PDEEstLem2a}
	\end{align}
	almost surely.
	Applying the Cauchy-Schwarz and the Young's inequalities to 
	\eqref{Eq:PDEEstLem2a} produces
	\begin{align*}
		&k \lambda \| \bfE_T(\om,\cdot) \|^2_{L^2(\pa D)} 
		\leq k \lambda \delta_2 \| \bfE(\om,\cdot) \|^2_{L^2(D)} 
		+ \frac{1}{4 k \lambda \delta_2} \| \bff(\om,\cdot) \|^2_{L^2(\pa D)},
	\end{align*}
	almost surely.
	Thus, \eqref{Eq:PDEEstLem2} holds.  Similarly, applying the Cauchy-Schwarz 
	and Young's inequalities to \eqref{Eq:PDEEstLem1a} along with the 
	properties of the random coefficient $\eta$ gives
	\begin{align*}
		& \| \curl \bfE(\om,\cdot) \|^2_{L^2(D)} \leq \left( 
		k^2(1+\veps)^2 + \delta_1 \right) \| \bfE(\om,\cdot) 
		\|^2_{L^2(D)} + \frac{1}{4 \delta_1} \| \bff(\om,\cdot) 
		\|^2_{L^2(D)}.
	\end{align*}
	Thus, \eqref{Eq:PDEEstLem1} holds.
\end{proof}

With the Rellich identities from Lemma \ref{Lem:RellichIDs} and estimates from 
Lemma \ref{Lem:PDEEstimate1} we are now able to prove stability estimates for 
solutions to \eqref{Eq:PDE1}--\eqref{Eq:PDE2}.  The following theorem gives 
the wavenumber-explicit estimates and is the main result of this section:
\begin{theorem} \label{Thm:PDEEstimate}
	Let $\bff: \Om \to H(\ddiv,D)$ and $\bfE: \Om \to \hat{\bcV}$ be the weak 
	solution to 
	\eqref{Eq:PDE1}--\eqref{Eq:PDE2} . Then for $0 < \veps < \frac{1}{32R \mu}
	$ chosen to satisfy $\veps(2 + \veps) < \gamma_0 := \min\left\{1, \frac{1}
	{8Rk}\right\}$, the following estimates hold, almost surely:
	\begin{align}
		&\| \bfE(\om,\cdot) \|^2_{L^2(D)}  +  \| \bfE(\om,\cdot) \|
		^2_{L^2(\pa D)} \label{Eq:PDEEstimate0a} \\
		&\qquad \qquad \leq C_0 \left(\frac{1}{k} + \frac{1}{k^2} \right)^2 
		\big( \| \bff(\om,\cdot) \|^2_{L^2(D)} + \| \ddiv 
		\bff(\om,\cdot) \|^2_{L^2(D)} \big), \notag \\
		&\| \curl \bfE(\om,\cdot) \|
		^2_{L^2(D)} + \| \curl \bfE(\om,\cdot) \|
		^2_{L^2(\pa D)} \label{Eq:PDEEstimate0b} \\
		& \qquad \qquad \leq C_0 \left(1 + \frac{1}{k} \right)^2 
		\big( \| \bff(\om,\cdot) \|^2_{L^2(D)} + \| \ddiv 
		\bff(\om,\cdot) \|^2_{L^2(D)} \big), \notag \\
		&\| \ddiv \bfE(\om,\cdot) \|^2_{L^2(D)}
		\leq C_0 \left(\frac{1}{k} + \frac{1}{k^2} \right)^2 
		\big( \| \bff(\om,\cdot) \|^2_{L^2(D)} + \| \ddiv 
		\bff(\om,\cdot) \|^2_{L^2(D)} \big), \label{Eq:PDEEstimate0c}
	\end{align}
	where $C_0$ is a positive constant independent of $k$, $\om$, and $\bfE$. 
	Furthermore, if $\bff \in \bL^2(\Om,\bH(\ddiv,D))$ and 
	$\bfE \in \bL^2(\Om,\hat{\bcV})$, then the following 
	estimates hold:
	\begin{align}
		\E \big( \| \bfE \|^2_{L^2(D)} \big) + \E \big( \| \bfE \|^2_{L^2(\pa 
		D)} \big) &\leq C_0 \left(\frac{1}{k} + \frac{1}{k^2} \right)^2 
		\mathcal{M}(\bff) \label{Eq:PDEEstimate1}, \\
		\E \big( \| \curl \bfE \|
		^2_{L^2(D)} \big) + \E \big( \| \curl \bfE \|
		^2_{L^2(\pa D)} \big)
		& \leq C_0 \left(1 + \frac{1}{k} \right)^2 
		\mathcal{M}(\bff), \label{Eq:PDEEstimate2} \\
		\E \big( \| \ddiv \bfE \|^2_{L^2(D)} \big) 
		& \leq C_0 \left(\frac{1}{k} + \frac{1}{k^2} \right)^2 
		\mathcal{M}(\bff), \label{Eq:PDEEstimate3}
	\end{align}
	where $\mathcal{M}(\bff)$ is defined as
	\begin{align*}
		\mathcal{M}(\bff) := \E \big( \| \bff \|^2_{L^2(D)} \big) + \E \big( 
		\| \ddiv \bff \|^2_{L^2(D)} \big).
	\end{align*}
\end{theorem}
 
\begin{proof}
	In this proof we fix $\om \in \Omega$ and the results will hold almost 
	surely.
	We demonstrate the proof for $\bfE(\om,\cdot) \in \bH^2(D)$.  The more 
	general result can be obtained by replacing $\bfE(\om,\cdot)$ with its 
	mollification $\bfE_\rho(\om,\cdot)$ and then letting $\rho \to 0$ after 
	the inequalities are obtained.
	Letting $\bfv = \curl \bfE(\om,\cdot) \times \bfx$ in \eqref{Eq:SesqForm} 
	yields
	\begin{align*}
		&2 \re a(\bfE(\om,\cdot),\bfv) = 2 \re \Big( 
		\big( \curl \bfE(\om,\cdot), \curl \bfv \big)_D - k^2 \big(\alpha(\om,
		\cdot)^2 \bfE(\om,\cdot), \bfv \big)_D\Big) \\
		& \qquad \qquad + 2 k \lambda \im \big \langle \bfE_T(\om,\cdot),
		\bfv_T \big \rangle_{\pa D}, \\
		& \qquad = 2 \re \Big( \big( \curl \bfE(\om,\cdot), \curl \bfv 
		\big)_D - k^2 \big(\bfE(\om,\cdot), \bfv \big)_D \\
		&\qquad \qquad - k^2 \Big(\vep 
		\eta(\om,\cdot)(2 + \vep \eta(\om,\cdot)) \bfE(\om,\cdot), \bfv 
		\Big)_D \Big)  + 2 k \lambda \im \big \langle \bfE_T(\om,\cdot),
		\bfv_T \big \rangle_{\pa D}.
	\end{align*}
	Applying \eqref{Eq:Rellich1} and \eqref{Eq:Rellich2} to the 
	above identity, rearranging the terms, and applying 
	\eqref{Eq:StarShape} yields the following:
	\begin{align} \label{Eq:PDEEstimate1a}
		& k^2 \| \bfE(\om,\cdot) \|_{L^2(D)}^2 +  \| 
		\curl \bfE(\om,\cdot) \|_{L^2(D)}^2 \\
		& = - k^2 \big \langle \bfx \cdot 
		\bfnu, | \bfE(\om,\cdot)|^2 \big \rangle_{\pa D} - \big \langle \bfx 
		\cdot \bfnu, | \curl \bfE(\om,\cdot)|^2 \big \rangle_{\pa D} \notag 
		\\
		& \quad - 2 \re \Big( k^2 \big(\bfx \, \ddiv 
		\bfE(\om,\cdot), \bfE(\om,\cdot) \big)_D + k^2 \big \langle \bfx 
		\times \bfE(\om,\cdot), \bfE(\om,\cdot) 
		\times \bfnu \big \rangle_{\pa D} 
		\notag \\
		& \quad + 2 k^2 \re \big(\vep \eta(\om,\cdot)(2 + \vep \eta(\om,
		\cdot)) \bfE(\om,\cdot), 
		\bfv \big)_D - 2 k \lambda \im \big \langle  \bfE_T(\om,\cdot),\bfv_T 
		\big \rangle_{\pa D} \notag \\ 
		& \quad + 2 \re a\big(\bfE(\om,\cdot),\bfv\big) \notag \\
		& \leq - c_0 k^2 \| \bfE(\om,\cdot) \|_{L^2(\pa D)}^2 
		- c_0 \| \curl \bfE(\om,\cdot) \|_{L^2(\pa D)}^2  \notag\\
		& \quad - 2 \re \Big( k^2 \big(\bfx \, \ddiv
		\bfE(\om,\cdot), \bfE(\om,\cdot) \big)_D + k^2 \big \langle \bfx 
		\times \bfE(\om,\cdot), \bfE(\om,\cdot) 
		\times \bfnu \big \rangle_{\pa D} \Big)  
		\notag \\
		& \quad + 2 k^2 \gamma_0 \re \big(\bfE(\om,\cdot), 
		\bfv \big)_D - 2 k \lambda \im \big \langle \bfE_T(\om,\cdot),\bfv_T 
		\big \rangle_{\pa D} + 2 \re a\big(\bfE(\om,\cdot),\bfv\big). \notag 
	\end{align}
	We can use the identities $\ba = \ba_T + (\ba\cdot \bfnu) 
	\ba$ and $(\ba \times \bfb) \cdot (\bfc \times \bfd) = (\ba 
	\cdot \bfc)(\bfb \cdot \bfd) - (\ba \cdot \bfd)(\bfb \cdot 
	\bfc)$ to show
	\begin{align} \label{Eq:PDEEstimate1b}
		&2k^2 \re \big \langle \bfx \times \bfE(\om,\cdot), \bfE(\om,\cdot) 
		\times \bfnu \big \rangle_{\pa D} \\ 
		& \qquad = 
		2k^2 \big\langle \bfx_T \cdot \bfE_T(\om,\cdot), \bfE(\om,\cdot) \cdot
		\bfnu \big\rangle_{\pa D} - 2k^2 \big \langle \bfx \cdot \bfnu, |
		\bfE(\om,\cdot) \times \bfnu|^2 \big \rangle_{\pa D} \notag
		\\
		& \qquad = 2k^2 \big\langle \bfx_T \cdot 
		\bfE_T(\om,\cdot), \bfE(\om,\cdot) \cdot \bfnu \big\rangle_{\pa D} - 
		2k^2 \big \langle \bfx \cdot \bfnu, |\bfE_T(\om,\cdot)|^2 \big 
		\rangle_{\pa D}. \notag
	\end{align}
	Note the last step was possible due to the following identity:
	\begin{align*}
		| \bfE_T(\om,\cdot)|^2 = | (\bfnu \times \bfE(\om,\cdot)) \times 
		\bfnu|^2 = | 
		\bfE(\om,\cdot) \times \bfnu|^2 - \big((\bfE(\om,\cdot) \times \bfnu) 
		\cdot \bfnu 
		\big)^2 = |\bfE(\om,\cdot) \times \bfnu|^2.
	\end{align*}
	By rearranging the terms of \eqref{Eq:PDEEstimate1a} and 
	applying \eqref{Eq:PDEEstimate1b} and  \eqref{Def:WeakForm} we find the 
	following:
	\begin{align*}
		&k^2 \| \bfE(\om,\cdot) \|^2_{L^2(D)} +\| \curl \bfE(\om,\cdot) \|
		^2_{L^2(D)} + c_0 k^2  \| \bfE(\om,\cdot) \|
		^2_{L^2(\pa D)} + c_0  \| \curl \bfE(\om,\cdot) \|
		^2_{L^2(\pa D)} \\
		& \qquad = -2k^2 \big\langle \bfx_T 
		\cdot \bfE_T(\om,\cdot), \bfE(\om,\cdot) \cdot \bfnu \big\rangle_{\pa 
		D} + 2k^2 \big \langle \bfx \cdot \bfnu, |\bfE_T(\om,\cdot)|^2 \big 
		\rangle_{\pa D} \\
		& \qquad \qquad + 2 k^2 \gamma_0 \re \big(\bfE(\om,\cdot), 
		\bfv \big)_D - 2 k \lambda \im \big \langle \bfE_T(\om,\cdot),\bfv_T 
		\big \rangle_{\pa D} \\
		& \qquad \qquad - 2k^2 \re \big(\bfx \, \ddiv \bfE(\om,\cdot), 
		\bfE(\om,\cdot) \big)_D + 2 \re  \big(\bff(\om,\cdot), \bfv \big)_D.
	\end{align*} 
	We apply this identity, the Cauchy-Schwarz inequality, and the Young's 
	inequality to the above and establish the following inequality:
	\begin{align*}
		&k^2 \| \bfE(\om,\cdot) \|^2_{L^2(D)} +  \| \curl \bfE(\om,\cdot) \|
		^2_{L^2(D)} + c_0 k^2 \| \bfE(\om,\cdot) \|^2_{L^2(\pa D)} + c_0 
		\| \curl \bfE(\om,\cdot) \|^2_{L^2(\pa D)} \\
		& \qquad \leq Rk^2 \left( \frac{1}{\delta_1} \| 
		\bfE_T(\om,\cdot) \|^2_{L^2(\pa D)} + \delta_1 \| \bfE(\om,\cdot) \|
		^2_{L^2(\pa D)}\right)  + 2Rk^2\left(
		\| \bfE_T(\om,\cdot) \|^2_{L^2(\pa D)}\right) \\
		& \qquad \qquad + k \lambda R \left( \frac{1}{\delta_2}
		\| \bfE_T(\om,\cdot) \|^2_{L^2(\pa D)} + \delta_2
		\| \curl \bfE(\om,\cdot) \|^2_{L^2(\pa D)}\right) \\
		& \qquad \qquad + R k^2 \gamma_0 \left( \frac{1}{\delta_3} 
		\| \bfE(\om,\cdot) \|_{L^2(D)}^2 + \delta_3 \| \curl 
		\bfE(\om,\cdot) \|_{L^2(D)}^2 \right) \\
		&\qquad \qquad + R \left( \frac{1}{\delta_4} \| \bff(\om,\cdot) \|
		^2_{L^2(D)} + \delta_4  \| \curl \bfE(\om,\cdot) \|^2_{L^2(D)} 
		\right) \\
		& \qquad \qquad - 2k^2 \re \big(\bfx \, \ddiv 
		\bfE(\om,\cdot), \bfE(\om,\cdot) \big)_D.
	\end{align*}
	Next, we choose $\delta_1 = \frac{c_0}{2 R}$, $\delta_2 = \frac{c_0}{2 
	k \lambda R}$, $\delta_3 = \frac{1}{k}$, and $\delta_4 = \frac{1}{4 R}$. 
	We also use the fact that $\gamma_0 \leq 1$ 
	implies $\veps \leq \frac{1}{2}$. Thus, from \eqref{Eq:PDEEstLem2} we 
	find
	\begin{align*}
		&k^2 \left(1 - Rk \gamma_0 \right) \| \bfE(\om,\cdot) \|
		^2_{L^2(D)} + \left( \frac{3}{4} - Rk \gamma_0 
		\right) \| \curl \bfE(\om,\cdot) \|^2_{L^2(D)} \\
		& \qquad \qquad \qquad + \frac{c_0 k^2}{2} \| \bfE(\om,\cdot) \|
		^2_{L^2(\pa D)}+ \frac{c_0}{2} \|\curl \bfE(\om,\cdot) \|
		^2_{L^2(\pa D)}\\
		& \qquad \leq \frac{2Rk^2}{c_0} \big(R + c_0 + \lambda^2 R
		\big) \| \bfE_T(\om,\cdot) \|^2_{L^2(\pa D)} -2k^2 \re 
		\big(\bfx \ddiv \bfE(\om,\cdot), \bfE(\om,\cdot) \big)_D\\
		& \qquad \qquad \qquad + 4R^2 \| \bff(\om,\cdot) \|
		^2_{L^2(D)} \\
		& \qquad \leq \frac{2Rk^2}{c_0} \big(R + c_0 + \lambda^2 R
		\big) \Big( \delta_5 \| \bfE(\om,\cdot) \|^2_{L^2(D)}+ 
		\frac{1}{4 k^2 \lambda^2 \delta_5} \| \bff(\om,\cdot) \|^2_{L^2(D)} 
		\Big)\\
		& \qquad \qquad \qquad -2k^2 \re \big(\bfx \ddiv \bfE(\om,\cdot), 
		\bfE(\om,\cdot) \big)_D + 4R^2 \| \bff(\om,\cdot) \|^2_{L^2(D)}.
	\end{align*}
	By choosing $\delta_5 = \frac{1}{8R} \left( \frac{c_0}{R + c_0 + 
	\lambda^2 R} \right)$ and rearranging the terms of this 
	inequality we find
	\begin{align} \label{Eq:PDEEstimate1c}
		&k^2 \left(\frac{3}{4} - Rk \gamma_0 \right) \| 
		\bfE(\om,\cdot) \|^2_{L^2(D)} + \left( \frac{3}{4} - Rk 
		\gamma_0 \right) \| \curl \bfE(\om,\cdot) \|^2_{L^2(D)}\\ 
		& \qquad \qquad \qquad + \frac{c_0 k^2}{2}  \| \bfE(\om,\cdot) \|
		^2_{L^2(\pa D)} + \frac{c_0}{2} \|\curl \bfE(\om,\cdot) \|
		^2_{L^2(\pa D)} \notag \\
		& \quad \leq - 2k^2 \re \big(\bfx \ddiv \bfE(\om,\cdot), 
		\bfE(\om,\cdot)\big)_D + \frac{4 R^2}{\lambda^2 c_0^2}\big(R + c_0 + 
		\lambda^2 R \big)^2 \| \bff(\om,\cdot) \|^2_{L^2(D)} \notag \\
		& \qquad \qquad \qquad + 4R^2 \| \bff(\om,\cdot) \|^2_{L^2(D)} 
		. \notag
	\end{align}
	To deal with the $\ddiv \bfE$ term we apply the divergence operator to 
	both sides of the PDE \eqref{Eq:PDE1}.  This yields
	\begin{align*}
		\ddiv\big( \alpha(\om,\cdot)^2 \bfE(\om,\cdot)\big) = -k^{-2} \ddiv 
		\bff(\om,\cdot),
	\end{align*}
	for almost surely.  Expanding the divergence on the 
	left-hand side of the equation and rearranging the terms gives
	\begin{align} \label{Eq:DivExp}
		\ddiv \bfE(\om,\cdot) = - 2\frac{\nabla \alpha(\om,\cdot)}{\alpha(\om,
		\cdot)} \cdot \bfE(\om,\cdot) - \frac{1}{\big(k \alpha(\om,\cdot)
		\big)^2} \ddiv \bff(\om,\cdot).
	\end{align}
	Here we have used the definition of $\alpha(\om,\cdot)$ and the fact that 
	$0< \veps < 1$ to divide both sides of the equation by $\alpha(\om,\cdot)
	$. We again note that $\gamma_0 \leq 1$ implies $\veps \leq \frac{1}{2}$ 
	and use the definition of $\alpha(\om,\cdot)$ to find
	\begin{align} \label{Eq:PDEEstimate1d}
		&-2k^2 \re \big(\bfx \ddiv \bfE(\om,\cdot), \bfE(\om,\cdot) \big)_D\\
		& \qquad \qquad= 2k^2 \re \bigg( 2\big( \alpha(\om,\cdot)^{-1} 
		(\nabla \alpha(\om,\cdot)) \cdot \bfE(\om,\cdot), \bfx \cdot \bfE(\om,
		\cdot)\big)_D  \notag\\
		&   \qquad \qquad +  \frac{1}{k^2} \big( 
		\bfx \alpha(\om,\cdot)^{-2} \ddiv \bff(\om,\cdot), \bfE(\om,\cdot) 
		\big)_D  \bigg) \notag \\ 
		&  \qquad \leq 8 R k^2 \mu \veps \| \bfE(\om,\cdot) \|
		^2_{L^2(D)}+ \frac{4R}{\delta_6} \| \ddiv \bff(\om,\cdot)\|^2_{L^2(D)} 
		+ 4R \delta_6 \| \bfE(\om,\cdot)\|^2_{L^2(D)}. \notag
	\end{align}
	We apply \eqref{Eq:PDEEstimate1d} with $\delta_6 = \frac{k^2}{16R}$ to 
	\eqref{Eq:PDEEstimate1c} and use $\veps \leq \frac{1}{32R\mu}$ to obtain
	\begin{align*}
		&k^2 \left(\frac{1}{4} - Rk \gamma_0 \right) \| 
		\bfE(\om,\cdot) \|^2_{L^2(D)} + \left( \frac{3}{4} - Rk \gamma_0 
		\right) \| \curl \bfE(\om,\cdot) \|^2_{L^2(D)} \\ 
		&   \qquad \qquad + \frac{c_0 k^2}{2} \| \bfE(\om,\cdot) \|
		^2_{L^2(\pa D)} \big)+ \frac{c_0}{2} \|\curl \bfE(\om,\cdot) \|
		^2_{L^2(\pa D)}\\
		& \qquad \leq \frac{4 R^2}{\lambda^2 c_0^2}\big(R + c_0 + \lambda^2 R 
		\big)^2 \| \bff(\om,\cdot) \|^2_{L^2(D)}\\
		& \qquad \qquad + 4R^2 \| \bff(\om,\cdot) \|^2_{L^2(D)} 
		+ \frac{64R^2}{k^2} \| \ddiv \bff(\om,\cdot)\|^2_{L^2(D)}.
	\end{align*}
	Since $\gamma_0 \leq \frac{1}{8Rk}$ we have
	\begin{align*}
		&\frac{k^2}{8} \| \bfE(\om,\cdot) \|^2_{L^2(D)} + \frac{5}{8} 
		\| \curl \bfE(\om,\cdot) \|^2_{L^2(D)} \\
		&\qquad \qquad \qquad+ \frac{c_0 k^2}{2}
		\| \bfE(\om,\cdot) \|^2_{L^2(\pa D)}+ \frac{c_0}{2}\|\curl 
		\bfE(\om,\cdot)\|^2_{L^2(\pa D)}\\
		& \qquad \leq \frac{4 R^2}{\lambda^2 c_0^2}\big(R + c_0 + 
		\lambda^2 R  \big)^2 \| \bff(\om,\cdot) \|^2_{L^2(D)}  \\
		& \qquad \qquad \qquad  + 4R^2 \| \bff(\om,\cdot) \|^2_{L^2(D)} 
		+ \frac{64R^2}{k^2} \| \ddiv \bff(\om,\cdot)\|^2_{L^2(D)}.
	\end{align*}
	\eqref{Eq:PDEEstimate0a} and \eqref{Eq:PDEEstimate0b} follow directly from 
	the previous inequality.\\
	From \eqref{Eq:DivExp} we find
	\begin{align*}
		\| \ddiv \bfE(\om,\cdot) \|^2_{L^2(D)}\leq 4\mu\veps\| 
		\bfE(\om,\cdot) \|^2_{L^2(D)}+ 4k^{-2}\| \ddiv \bff(\om,\cdot) \|
		^2_{L^2(D)}.
	\end{align*}
	Here we have used the fact that $\gamma_0 \leq 1$ implies $\veps \leq 
	\frac{1}{2}$.  Using $\veps \leq \frac{1}{32R\mu}$ and applying 
	\eqref{Eq:PDEEstimate0a} yields \eqref{Eq:PDEEstimate0c}.  Since 
	\eqref{Eq:PDEEstimate0a}--\eqref{Eq:PDEEstimate0c} hold almost surely and 
	$C_0$ does not depend on $\om$, then if $\bfE \in \bL^2(\Om,\hat{\bcV})$, 
	\eqref{Eq:PDEEstimate1}--\eqref{Eq:PDEEstimate3} hold.
\end{proof}

\begin{theorem} \label{Thm:PDEWellPosedness}
	Let $\bff: \Om \to \bH(\ddiv,D)$, then 
	there exists a unique solution $\bfE$  for the variational problem \eqref{Eq:WeakForm}.
\end{theorem}
\begin{proof}
For fixed $\om \in \Om$, we give a sketch of the proof for the well-posedness in the following and refer the readers to \cite{Monk_03} for more details.
First, the function space $\bcV$ adopts the Helmholtz decomposition 
$$ \bcV = \bcV_0 + \nabla S, $$
where
$$ \bcV_0 = \{ \mathbf{u} \in \bcV \;|\; (\alpha \mathbf{u}, \nabla p )_D =0 \; \forall p \in S  \} \quad \mbox{and} \quad S =  H_0^1(D). $$
Then the function space $\bcV_0$ is compactly embedded in $\bL^2(D)$ (cf. Theorem 4.7, \cite{Monk_03}). 
Correspondingly, the electric field is decomposed as $\bfE(\ome,\cdot) = \bfE_0(\ome,\cdot)  + \nabla p(\ome,\cdot) $, where $\bfE_0 \in \bcV_0$ and $p \in S$. 
By decomposing $\bfv$ as $\bfv=\bfv_0 + \nabla q$, the variational formulation \eqref{Eq:WeakForm} becomes
\begin{align}\label{eqn:WeakFromHD}
&a\big(\bfE_0(\om,\cdot),\bfv_0\big) + a\big( \nabla p(\ome,\cdot) ,\bfv_0\big) -k^2(\alpha(\ome,\cdot) \nabla p(\om,\cdot),\nabla q )_D  \\ & \qquad \qquad \qquad = (\bff(\om,\cdot), \bfv_0)_D + (\bff(\om,\cdot), \nabla q)_D. \notag
\end{align}
The Lax-Milgram lemma ensures that there exists a unique solution $p \in S$ for the variational problem
$$ -k^2(\alpha(\ome,\cdot) \nabla p(\om,\cdot), \nabla q )_D =  (\bff(\om,\cdot), \nabla q)_D \quad q\in S.  $$ 
In addition,
$$ ||\nabla p(\om,\cdot)||_{\bL^2(D) } \le c ||\bff||_{\bL^2(D) }  $$
for some positive constant $c$. As such \eqref{eqn:WeakFromHD} reduces to
\begin{equation}\label{eqn:WeakFromRD}
a\big(\bfE_0(\om,\cdot),\bfv_0\big)  = (\bff(\om,\cdot), \bfv_0)_D - a\big( \nabla p(\ome,\cdot) ,\bfv_0\big) \quad \forall\bfv_0\in\bcV_0.  
\end{equation}
It can be shown that $a\big(\bfE_0,\bfE_0\big)$ satisfies a G\"{a}rding type inequality over the space $\bcV_0$ (cf. Lemma 4.10, \cite{Monk_03}). By the compact embedding of $\bcV_0$ in $\bL^2(D)$
and the uniqueness of the solution via unique continuation,  the Fredholm Alternative implies that there exists a unique solution $\bfE_0(\om,\cdot)$ of \eqref{eqn:WeakFromRD}.
Hence the well-posedness of \eqref{eqn:WeakFromHD} follows and its solution is given by $\bfE(\om,\cdot) =\bfE_0(\om,\cdot)+\nabla p(\om,\cdot)$.
This proves that the variational problem \eqref{Eq:WeakForm} attains solution almost surely. The uniqueness of the solution holds in the sense
that two functions $\bfE_1=\bfE_2$ ($\Omega \to \bcV$) if the measure of the set $\{\, \om \in \Om \,|\, \bfE_1(\om,\cdot)\neq\bfE_2(\om,\cdot) \, \}$ is zero.

\end{proof}

\begin{remark}
If $\bbf\in L^2(\Ome,\bH^1(\Div,D))$, then the unique solution $\bfE$ of \eqref{Eq:PDE1}--\eqref{Eq:PDE2}
also satisfies estimates \eqref{Eq:PDEEstimate1}--\eqref{Eq:PDEEstimate3}. 
\end{remark}

\section{Multi-modes representation of the solution and its finite modes 
approximation} \label{sec:Multi-Modes}
The goal of this section is to show that  the solution to  
\eqref{Eq:PDE1}--\eqref{Eq:PDE2} admits a power series expansion in terms of the 
perturbation parameter  $\veps$ for sufficiently small $\veps$.  
This power series expansion will be used to design the efficient algorithm for solving the time-harmonic Maxwell's equations in random media.  
We begin by assuming this expansion formally, followed by analysis which justifies the expansion.

Assume that $\bfE^\veps$ solves \eqref{Eq:PDE1}--\eqref{Eq:PDE2} and takes the 
following form:
\begin{align} \label{Eq:MultiModes}
	\bfE^\veps = \sum_{n=0}^\infty \veps^n \bfE_n.
\end{align}
We call \eqref{Eq:MultiModes} the multi-modes expansion of the solution to 
\eqref{Eq:PDE1}--\eqref{Eq:PDE2}, which will be justified later in this 
section. Due to the scalability of the Maxwell's equations, without loss of 
generality, in the rest of the paper, we assume that $k \geq 1$ and $D$ lies in the unit disk such that 
$D \subset B_1(\textbf{0})$.

By substituting \eqref{Eq:MultiModes} into \eqref{Eq:PDE1}, we see that
\begin{align*}
	\bff &= \curl \curl \bfE^\veps - k^2 \alpha^2 \bfE^\veps \\
	&= \sum_{n=0}^\infty \veps^n \Big( \curl \curl \bfE_n - k^2 \big(1 + 
	2\veps \eta + \veps^2 \eta^2 \big) \bfE_n \Big)\\
	&= \curl \curl \bfE_0 - k^2 \bfE_0 + \veps \big(\curl \curl \bfE_1 - k^2 
	\bfE_1 - 2k^2 \eta \bfE_0\big)\\
	& \qquad + \sum_{n = 2}^\infty \veps^n \big( \curl \curl \bfE_n - k^2 
	\bfE_n - 2k^2 \eta \bfE_{n-1} - k^2\eta^2 \bfE_{n-2} \big).
\end{align*}
By matching coefficients of $\veps^n$ order terms we see that each 
individual mode function satisfies the following equations in $\Om \times D
$:
\begin{align}
	&\bfE_{-1} \equiv \mathbf{0}, \label{Eq:ModePDE1} \\
	& \curl \curl \bfE_0 - k^2 \bfE_0 = \bff, \label{Eq:ModePDE2} \\
	& \curl \curl \bfE_n - k^2 \bfE_n = 2k^2 \eta \bfE_{n-1} + k^2\eta^2 
	\bfE_{n-2} \qquad \forall n \geq 1 \label{Eq:ModePDE3}.
\end{align}
Similarly, each mode function satisfies the following boundary condition on $\Om \times \pa D$:
\begin{align}
	&\curl \bfE_n \times \bfnu - \i k \lambda \bfE_n = \mathbf{0} \qquad 
	\forall  n \geq 0. \label{Eq:ModePDE4}
\end{align}

An important feature of the multi-modes expansion of the solution 
\eqref{Eq:MultiModes} is that each mode function $\bfE_n$ satisfies
the same time-harmonic Maxwell's system with recursively defined random sources.
This is the key to develop the efficient algorithm discussed later in Section \ref{sec:Numerical_Procedure}.
The following theorem gives stability estimates for each mode function $\bfE_n
$. This theorem is a consequence of Theorem \ref{Thm:PDEEstimate} with $
\veps = 0$ along with the recursive relationship defined in 
\eqref{Eq:ModePDE3}.
\begin{theorem} \label{Thm:MultiModes}
	Let $\bff: \Om \to \bH(\ddiv,D)$. Then for each $n \geq 0$, 
	there exists a 
	unique solution $\bfE_n: \Om \to \bcV$ satisfying 
	\eqref{Eq:ModePDE2} 
	and \eqref{Eq:ModePDE4} for $n = 0$ and \eqref{Eq:ModePDE3} and 
	\eqref{Eq:ModePDE4} for $n \geq 1$ (in the sense of Definition 
	\ref{Def:WeakForm}).  Moreover, for $n \geq 0$, $\bfE_n$ satisfies the 
	following stability estimates:
	\begin{align}
		&\| \bfE_n(\om,\cdot) \|^2_{L^2(D)} +  \| \bfE_n(\om,\cdot) \|
		^2_{L^2(\pa D)} 
		\label{Eq:ModeEst0a} \\
		& \qquad \qquad \qquad\leq C(n,k) \left(\frac{1}{k} + \frac{1}{k^2} 
		\right)^2 \left(\| \bff(\om,\cdot) \|^2_{L^2(D)} + \| 
		\ddiv \bff(\om,\cdot) \|^2_{L^2(D)} \right), \notag \\
		&  \| \curl \bfE_n(\om,\cdot) \|
		^2_{L^2(D)} + \| \curl \bfE_n(\om,\cdot) \|
		^2_{L^2(\pa D)} \label{Eq:ModeEst0b} \\
		& \qquad \qquad \qquad \leq C(n,k) \left(1 + \frac{1}{k} \right)^2 
		\left(\| \bff(\om,\cdot) \|^2_{L^2(D)} + \| 
		\ddiv \bff(\om,\cdot) \|^2_{L^2(D)} \right), \notag \\
		& \| \ddiv \bfE_n(\om,\cdot) \|^2_{L^2(D)} \label{Eq:ModeEst0c} \\
		& \qquad \qquad \qquad \leq C(n,k) 
		\left(\frac{1}{k} + \frac{1}{k^2} \right)^2 \left(\| \bff(\om,\cdot) 
		\|^2_{L^2(D)} + \| \ddiv \bff(\om,\cdot) \|^2_{L^2(D)} \right), 
		\notag
	\end{align}
	where
	\begin{align*}
		C(0,k) := C_0, \qquad C(n,k) := 7^{2n-1}C_0^{n+1}(1+k)^{2n}(1+\mu)^{2n} 
		\qquad\forall n \geq 1.
	\end{align*}
	Furthermore, if $\bff \in \bL^2(\Om,\bH(\ddiv,D))$ and $\bfE_n \in 
	\bL^2(\Om,\bH^1(D))$, then the following estimates 
	hold
	\begin{align}
		&\E \big( \| \bfE_n \|^2_{L^2(D)} \big) + \E \big( \| \bfE_n \|
		^2_{L^2(\pa D)} \big) 
		\label{Eq:ModeEst1} \\
		& \qquad \qquad \qquad\leq C(n,k) \left(\frac{1}{k} + \frac{1}{k^2} 
		\right)^2 \mathcal{M}(\bff), \notag \\
		& \E \big( \| \curl \bfE_n \|
		^2_{L^2(D)} \big) + \E \big( \| \curl \bfE_n \|
		^2_{L^2(\pa D)} \big) \label{Eq:ModeEst2} \\
		& \qquad \qquad \qquad \leq C(n,k) \left(1 + \frac{1}{k} \right)^2 
		\mathcal{M}(\bff), \notag \\
		&\E \big( \| \ddiv \bfE_n \|^2_{L^2(D)} \big) \leq C(n,k) 
		\left(\frac{1}{k} + \frac{1}{k^2} \right)^2 \mathcal{M}(\bff), 
		\label{Eq:ModeEst3}
	\end{align}
	where
	\begin{align*}
		\mathcal{M}(\bff) := \E \big( \| \bff \|^2_{L^2(D)} \big) + \E \big( 
		\| \ddiv \bff \|^2_{L^2(D)} \big).
	\end{align*}
\end{theorem}
\begin{proof}
Existence and uniqueness of each mode function $\bfE_n$ is a consequence of 
stability estimates \eqref{Eq:ModeEst0a}--\eqref{Eq:ModeEst0c} and an argument 
similar to the proof of Theorem \ref{Thm:PDEWellPosedness}.  Also, 
\eqref{Eq:ModeEst1}--\eqref{Eq:ModeEst3} are direct consequences of 
\eqref{Eq:ModeEst0a}--\eqref{Eq:ModeEst0c}.  Thus, we focus our attention on 
proving \eqref{Eq:ModeEst0a}--\eqref{Eq:ModeEst0c}.

For the rest of the proof we fix $\om \in \Om$ and all estimates and 
identities will hold almost surely.
By definition each mode function $\bfE_n$ satisfies \eqref{Eq:PDE1}--\eqref{Eq:PDE2}
with $\veps = 0$ in the left-hand side Maxwell's differential 
operator and source function $\bff$ defined through a recursive relationship 
given in \eqref{Eq:ModePDE1}--\eqref{Eq:ModePDE3}.  Hence we can apply Theorem 
\ref{Thm:PDEEstimate} to obtain stability estimates for each mode function. 
Beginning with the first mode function $\bfE_0$, Theorem \ref{Thm:PDEEstimate} 
yields
\begin{align*}
	&\| \bfE_0(\om,\cdot) \|^2_{L^2(D)}+ \| \bfE_0(\om,\cdot) \|^2_{L^2(\pa 
	D)} \\ 
	& \qquad \qquad \leq C_0 \left( \frac{1}{k} + \frac{1}{k^2} \right)^2 
	\left(\| \bff(\om,\cdot) \|^2_{L^2(D)} + \| \ddiv \bff(\om,\cdot) \|
	^2_{L^2(D)} \right),
	 \\
	&\| \curl \bfE_0(\om,\cdot) \|^2_{L^2(D)}+ \| \curl \bfE_0(\om,\cdot) \|
	^2_{L^2(\pa D)} \\ 
	& \qquad \qquad \leq C_0 \left( \frac{1}{k} + \frac{1}{k^2} \right)^2 
	\left(\| \bff(\om,\cdot) \|^2_{L^2(D)} + \| \ddiv \bff(\om,\cdot) \|
	^2_{L^2(D)} \right),
	\\
	&\| \ddiv \bfE_0(\om,\cdot) \|^2_{L^2(D)} \\
	& \qquad \qquad \leq C_0 \left( \frac{1}{k} + \frac{1}{k^2} \right)^2 
	\left(\| \bff(\om,\cdot) \|^2_{L^2(D)} + \| \ddiv \bff(\om,\cdot) \|
	^2_{L^2(D)} \right).
\end{align*} 
Thus, \eqref{Eq:ModeEst0a}--\eqref{Eq:ModeEst0c} are verified for $n = 0$. We will prove \eqref{Eq:ModeEst0a}--\eqref{Eq:ModeEst0c} for $n > 0$ by induction on $n$.  Assume \eqref{Eq:ModeEst0a}--\eqref{Eq:ModeEst0c} hold for all $n \leq \ell - 1$.  Then for $\bfE_{\ell}$ we apply Theorem \ref{Thm:PDEEstimate} and find
\begin{align*}
	&\| \bfE_\ell(\om,\cdot) \|^2_{L^2(D)} + \| \bfE_\ell(\om,\cdot) \|
	^2_{L^2(\pa D)} \\
	& \quad \leq 2 C_0 \left( \frac{1}{k} + \frac{1}{k^2} \right)^2 \left(
	4k^4 \| \bfE_{\ell-1}(\om,\cdot) \big\|^2_{L^2(D)} + 4k^4 \|\ddiv 
	\big(\eta(\om,\cdot) \bfE_{\ell-1}(\om,\cdot)\big) \big\|^2_{L^2(D)} \right. \\
	& \quad \qquad \left. +  k^4 \| \bfE_{\ell-2}(\om,\cdot) \|^2_{L^2(D)} 
	+ k^4 \big\|\ddiv \big(\eta(\om,\cdot) \bfE_{\ell-2}(\om,\cdot)\big) \big\|^2_{L^2(D)}
	\right)\\
	& \quad \leq 2 C_0 k^4 \left( \frac{1}{k} + \frac{1}{k^2} \right)^2 
	\left( (4 + 8\mu^2) \| \bfE_{\ell-1}(\om,\cdot) \|^2_{L^2(D)}+ 8 
	\|\ddiv \bfE_{\ell-1}(\om,\cdot) \|^2_{L^2(D)}\right. \\
	& \quad \qquad \left. +  (1 + 4\mu^2) \| \bfE_{\ell-2}(\om,\cdot) \|
	^2_{L^2(D)} + 2 \|\ddiv \bfE_{\ell-2}(\om,\cdot) \|^2_{L^2(D)}\right).
\end{align*}
Here we have used the properties of $\eta$ which yield the following inequalities
\begin{align*}
	\big| \ddiv \big(\eta(\om,\cdot) \bfE(\om,\cdot)\big) \big|^2 \leq 2 | \ddiv \bfE(\om,
	\cdot) |^2 + 2\mu^2 | \bfE(\om,\cdot) |^2, \\
	\big| \ddiv \big(\eta(\om,\cdot)^2 \bfE(\om,\cdot)\big) \big|^2 \leq 2 | \ddiv \bfE(\om,\cdot) |^2 + 4\mu^2 | \bfE(\om,\cdot) |^2,
\end{align*}
almost surely.
From the inductive hypothesis, we have
\begin{align*}
	&\| \bfE_\ell(\om,\cdot) \|^2_{L^2(D)} + \| \bfE_\ell(\om,\cdot)\|^2_{L^2(\pa D)} \\
	& \quad \leq 2 C_0 k^4 \left( \frac{1}{k} + \frac{1}{k^2} \right)^2 
	\left( (12 + 8\mu^2) C(\ell-1,k) \left( \frac{1}{k} + \frac{1}{k^2} 	\right)^2 \right. \\
	& \qquad  \left. +  (1 + 4\mu^2)  C(\ell-2,k) \left( \frac{1}{k} + 
	\frac{1}{k^2} \right)^2 \right) \left(\| \bff(\om,\cdot) \|^2_{L^2(D)} + 
	\| \ddiv \bff(\om,\cdot) \|^2_{L^2(D)} \right) \\
	& \quad \leq 24 C_0 (1+k)^2 (1 + \mu)^2 C(\ell-1,k) \left(1 + 
	\frac{C(\ell-2,k)}{C(\ell-1,k)} \right)\left( \frac{1}{k} + \frac{1}{k^2}	\right)^2 \\
	& \qquad \times \left(\| \bff(\om,\cdot) \|^2_{L^2(D)} + 
	\| \ddiv \bff(\om,\cdot) \|^2_{L^2(D)} \right) 
\end{align*}
\begin{align*}
	& \quad \leq 48 C_0 (1+k)^2 (1 + \mu)^2 C(\ell-1,k) \left( \frac{1}{k} + 	\frac{1}{k^2} \right)^2 \\
	& \qquad \times \left(\| \bff(\om,\cdot) \|^2_{L^2(D)} + \| \ddiv 
	\bff(\om,\cdot) \|^2_{L^2(D)} \right) \\
	& \quad \leq C(\ell,k) \left( \frac{1}{k} + \frac{1}{k^2} \right)^2 
	\left(\| \bff(\om,\cdot) \|^2_{L^2(D)} + \| \ddiv \bff(\om,\cdot) \|^2_{L^2(D)} \right).
\end{align*}
Here we have used the definition of $C(\ell,k)$ which implies $
\frac{C(\ell-2,k)}{C(\ell-1,k)} \leq 1$.  Thus, \eqref{Eq:ModeEst0a} holds for 
$n = \ell$.  By similar arguments \eqref{Eq:ModeEst0b} and 
\eqref{Eq:ModeEst0c} hold for $n = \ell$.  Therefore, by induction on $n$ 
\eqref{Eq:ModeEst0a}--\eqref{Eq:ModeEst0c} hold for all $n \geq 0$.
\end{proof}

For the rest of the analysis of this paper the mode functions $\bfE_n$ and the multi-modes solution $\bfE^\veps$ are assumed to be in $L^2(\Om,\bcV)$.
As an immediate consequence Theorem \ref{Thm:MultiModes},
we prove the following theorem which justifies the multi-modes expansion of the solution given in \eqref{Eq:MultiModes}.
\begin{theorem} \label{Thm:MultiModesSolution}
	Let $\bff \in \bL^2\big(\Om,\bH(\ddiv,D)\big)$ and $\bfE_n$ be the same as 
	those defined 
	in Theorem \ref{Thm:MultiModes}.  Then $\bfE^\veps$ defined by 
	\eqref{Eq:MultiModes} satisfies \eqref{Eq:PDE1}--\eqref{Eq:PDE2} (in 
	the sense of Definition \ref{Def:WeakForm}) provided that $\sigma := 
	7 \veps C_0^{\frac{1}{2}} (1 + k)(1 + \mu) < 1$.
\end{theorem}
\begin{proof}
	The proof consists of two parts: (i) To show that the infinite series defined 
	in \eqref{Eq:MultiModes} converges in $\bL^2(\Om,\hat{\bcV})$; (ii) To show 
	that the limit satisfies \eqref{Eq:PDE1}--\eqref{Eq:PDE2}.  To carry out 
	the analysis we begin by defining the finite sum
	\begin{align}
	\bfE^\veps_N = \sum_{n=0}^N \veps^n \bfE_n. \label{Eq:FiniteModes}
	\end{align}
	This finite sum will be referred to as the finite mode expansion of the solution.
	
	To show that $\bfE^\veps_N$ converges as $N \to \infty$, we show that it is 
	a Cauchy sequence.  For any fixed positive integer $p$ we apply Theorem 
	\ref{Thm:MultiModes}, Schwarz inequality, and a geometric sum to find 
	\begin{align*}
		&\E \big( \| \bfE^\veps_{N+p} - \bfE^\veps_{N} \|^2_{L^2(D)} \big) + 
		\E \big( \| \bfE^\veps_{N+p} - \bfE^\veps_{N} \|^2_{L^2(\pa D)} \big)
		\\ 
		&\qquad \qquad\leq p \sum_{n = N}^{N+p-1} \veps^{2n}
		\left( \E \big( \| \bfE_n \|^2_{L^2(D)} \big) + \E \big( \| \bfE_n \|
		^2_{L^2(\pa D)} \big)  \right) \\
		& \qquad \qquad \leq p \left(\frac{1}{k} + \frac{1}{k^2} \right)^2 
		\mathcal{M}(\bff) \sum_{n = N}^{N+p-1} \veps^{2n} C(n,k) \\
		& \qquad \qquad \leq C_0 p \left(\frac{1}{k} + \frac{1}{k^2} 
		\right)^2 \mathcal{M}(\bff) \sum_{n = N}^{N+p-1} \sigma^{2n} \\
		& \qquad \qquad \leq C_0 p \left(\frac{1}{k} + \frac{1}{k^2} 
		\right)^2 \mathcal{M}(\bff) \cdot \frac{\sigma^{2N}(1-\sigma^{2p})}{1-\sigma^2}.
	\end{align*}
	Similarly, the following also holds:
	\begin{align*}
		&\E \big( \| \curl (\bfE^\veps_{N+p} - \bfE^\veps_{N}) \|^2_{L^2(D)} 
		\big) + \E \big( \|  
		\curl (\bfE^\veps_{N+p} - \bfE^\veps_{N}) \|^2_{L^2(\pa D)} \big) \\
		& \qquad \qquad \leq C_0 p \left(1 + \frac{1}{k} 
		\right)^2 \mathcal{M}(\bff) \cdot \frac{\sigma^{2N}(1-\sigma^{2p})}{1-\sigma^2}.
	\end{align*}
	Thus, for $\sigma < 1$,
	\begin{align*}
		\lim_{N \to \infty} &\left[ \E \big( \| \bfE^\veps_{N+p} - \bfE^
		\veps_{N} \|^2_{L^2(D)} \big) + \E \big( \| \bfE^\veps_{N+p} - \bfE^
		\veps_{N} \|^2_{L^2(\pa D)} \big) \right. \\
		& \quad \left. \E \big( \| \curl (\bfE^\veps_{N+p} - \bfE^\veps_{N}) 
		\|^2_{L^2(D)} \big) + \E \big( \| \curl (\bfE^\veps_{N+p} - \bfE^
		\veps_{N}) \|^2_{L^2(\pa D)} \big) \right] = 0.
	\end{align*}
	Therefore, $\{ \bfE^\veps_N \}$ is a Cauchy sequence and we conclude that 
	the series \eqref{Eq:MultiModes} converges in $\bL^2(\Om,\hat{\bcV})$.
	
	Let $\bfE^\veps \in \bL^2(\Om,\hat{\bcV})$ be defined by 
	\eqref{Eq:MultiModes}.  By the 
	definitions of the mode functions $\bfE_n$ and the finite-mode expansion $
	\bfE^\veps_N$ we have
	\begin{align} \label{Eq:MultiModesSolution_a}
		\int_{\Om} a(\bfE^\veps_N , \bfv) \, dP = \int_\Om \Big( \big(\bff, 
		\bfv \big)_D - k^2 \veps^N \big(\eta(2+\veps \eta) \bfE_{N-1} + \eta^2 
		\bfE_{N-2}, \bfv \big)_D \Big)\, dP,
	\end{align}
	for all $\bfv \in \bL^2(\Om,\bcV)$. By using the Schwarz inequality and 
	Theorem \ref{Thm:MultiModes}, it follows that
	\begin{align*}
		&k^2 \veps^N \left| \int_{\Om} \big(\eta(2+\veps \eta) \bfE_{N-1} + 
		\eta^2 \bfE_{N-2}, \bfv \big)_D \, dP \right| \\
		& \qquad \leq 3k^2 \veps^N \Big( \big(\E (\| \bfE_{N-1} \|^2_{L^2(D)}) 
		\big)^{\frac{1}{2}} + \big(\E (\| \bfE_{N-2} \|^2_{L^2(D)})
		\big)^{\frac{1}{2}} \Big) \big(\E (\| \bfv \|^2_{L^2(D)})
		\big)^{\frac{1}{2}}\\
		& \qquad \leq 6k^2 \veps^N \left( \frac{1}{k} + \frac{1}{k^2} \right) 
		C(N-1,k)^\frac{1}{2}\mathcal{M}(\bff)^\frac{1}{2}\big(\E (\| \bfv \|^2_{L^2(D)})
		\big)^{\frac{1}{2}} \\
		& \qquad \leq 3k^2 \veps C_0^\frac{1}{2} \left( \frac{1}{k} + \frac{1}
		{k^2} \right) \mathcal{M}(\bff)^\frac{1}{2}\big(\E (\| \bfv \|^2_{L^2(D)})
		\big)^{\frac{1}{2}} \sigma^{N-1} \\
		& \qquad \longrightarrow 0 \, \mbox{ as } N \to \infty, \, \mbox{ 
		provided that } \sigma < 1.
	\end{align*}
	Thus, by letting $N \to \infty$ in \eqref{Eq:MultiModesSolution_a}
	\begin{align} \label{Eq:MultiModesSolution_b}
		\int_{\Om} a(\bfE^\veps , \bfv) \, dP = \int_\Om \big(\bff, 
		\bfv \big)_D \, dP \qquad \forall \bfv \in \bL^2(\Om,\bcV).
	\end{align}
 Therefore, we conclude that the multi-modes expansion $\bfE^\veps$ defined in \eqref{Eq:MultiModes} satisfies \eqref{Eq:PDE1}--\eqref{Eq:PDE2} in the sense of Definition \ref{Def:WeakForm} (cf. Remark \ref{rem2.1}).
\end{proof}
\begin{remark}
	To ensure the convergence of the multi-modes expansion of the solution in 
	\eqref{Eq:MultiModes}, Theorem \ref{Thm:MultiModesSolution} requires that $\veps = O ((1+k)^{-1}(1 + 
	\mu)^{-1})$. Compared with the scalar Helmholtz problem in random media in \cite{Feng_Lin_Lorton_15} where one requires $\veps = O( (1+k)^{-1})$,
	the new constraint on $\veps$ also depends on $\mu$, the upper bound for the gradient of the random field.
\end{remark}

To achieve a computable approximation to $\bfE^\veps$, we will use the  
truncated finite modes representation $\bfE^\veps_N$ defined by \eqref{Eq:FiniteModes}.
The following theorem gives the error associated with the finite modes 
representation. This theorem is a direct consequence of Theorems \ref{Thm:PDEEstimate} and \ref{Thm:MultiModes}.
\begin{theorem} \label{Thm:FiniteModesError}
	Let $\bff \in \bL^2\big(\Om,\bH(\ddiv,D)\big)$ and $\bfE_n$ be the same as 
	those defined 
	in Theorem \ref{Thm:MultiModes}. Then for $\bfE^\veps$ and $\bfE^\veps_N$ 
	defined in \eqref{Eq:MultiModes} and \eqref{Eq:FiniteModes}, respectively,
	the following estimates hold:
	\begin{align} \label{Eq:FiniteModesError_1}
		\E \big( \| \bfE^\veps - \bfE^\veps_N \|^2_{L^2(D)} \big)
		& \leq \frac{72 C_0 \sigma^{2N}}{343(1+k)^2(1+\mu)^2} 
		\left(1 + \frac{1}{k} \right)^4 \mathcal{M}(\bff),  \\
		\E \Big( \big\| \curl \big( \bfE^\veps - \bfE^\veps_N \big) \big\|
		_{L^2(D)}^2  \Big) & \leq \frac{72 C_0 \sigma^{2N}}{343(1+k)^2(1+\mu)^2} 
		\left(k + 1 \right)^4 \mathcal{M}(\bff), \label{Eq:FiniteModesError_2}
	\end{align}
	provided $\sigma = 7 \veps C_0^{\frac{1}{2}} (1 + k)(1 + \mu) < 1$.  Where
	\begin{align*}
		\mathcal{M}(\bff) := \E \big( \| \bff \|^2_{L^2(D)} \big) + \E \big( \| 
		\ddiv \bff \|^2_{L^2(D)} \big).
	\end{align*}
\end{theorem}
\begin{proof}
	By subtracting \eqref{Eq:MultiModesSolution_a} from 
	\eqref{Eq:MultiModesSolution_b} we obtain
	\begin{align*}
		\int_{\Om} a(\bfE^\veps - \bfE^\veps_N , \bfv) \, dP = \int_\Om  k^2 
		\veps^N \big(\eta(2+\veps \eta) \bfE_{N-1} + \eta^2 \bfE_{N-2}, \bfv 
		\big)_D \, dP.
	\end{align*}
	Thus $\bfE^\veps - \bfE^\veps_N$ satisfies 
	\eqref{Eq:PDE1}--\eqref{Eq:PDE2} with 
	$\bff = k^2 \veps^N \big(\eta(2+\veps \eta) \bfE_{N-1} + \eta^2 \bfE_{N-2} 
	\big)$. Applying Theorems \ref{Thm:PDEEstimate} and \ref{Thm:MultiModes} yields
	\begin{align*}
		&\E \big( \| \bfE^\veps - \bfE^\veps_N \|^2_{L^2(D)} \big) + 
		\E \big( \| \bfE^\veps - \bfE^\veps_N \|^2_{L^2(\pa D)} \big) \\
		& \qquad \leq 18 k^4 \veps^{2N} C_0 \left( \frac{1}{k} + \frac{1}{k^2} 
		\right)^2 \Big( \E \big( \| \bfE_{N-1} \|^2_{L^2(D)} \big) + \E \big( 
		\| \ddiv \bfE_{N-1} \|^2_{L^2(D)} \big) \\
		& \qquad \qquad \qquad + \E \big( \| \bfE_{N-2} \|^2_{L^2(D)} \big) + 
		\E \big( \| \ddiv \bfE_{N-2} \|^2_{L^2(D)} \big) \Big) \\
		& \qquad \leq 36  k^4 \veps^{2N} C_0 \left( \frac{1}{k} +
		 \frac{1}{k^2}\right)^4 \Big( C(N-1,k) + C(N-2,k) \Big) \mathcal{M}(\bff) \\
		& \qquad \leq 72  \veps^{2N} C_0 \left( 1 +
		 \frac{1}{k}\right)^4 C(N-1,k) \mathcal{M}(\bff) \\
		& \qquad \leq \frac{72C_0 \sigma^{2N}}{343(1+k)^2(1+\mu)^2} \left( 1 +
		 \frac{1}{k}\right)^4 \mathcal{M}(\bff).
	\end{align*}
	Thus, \eqref{Eq:FiniteModesError_1} holds.  The proof of \eqref{Eq:FiniteModesError_2} follows similarly.
\end{proof}

\section{Monte Carlo discontinuous Galerkin approximation of the mode $\bfE_n$} \label{sec:MCIP-DG}
We propose a Monte Carlo interior penalty discontinuous Galerkin (MCIP-DG) method for approximating 
$\E (\bfE_n$). The IP-DG method was developed in \cite{Feng_Wu_14}  for approximating the deterministic time-Harmonic Maxwell's equations with large wave number.
Compared with other discretization methods such as finite difference and finite  element methods, the IP-DG method is unconditionally stable.
We begin by introducing standard notations to describe the IP-DG method in the next subsection.
\subsection{DG notation}
First, we let $\mathcal{T}_h$ denote a family of quasi-uniform, shape-regular 
partitions of the spatial 
domain $D$ parameterized by $h$.  Typically, the domain $D$ is partitioned 
into tetrahedrons or parallelpipeds. Let $\mathcal{E}^I_h$ and $\mathcal{E}
^B_h$ denote the sets of interior faces and boundary faces of $\mathcal{T}_h$, 
respectively, defined as
\begin{align*}
	\mathcal{E}^I_h &:= \left\{\mbox{ Set of faces } \mathcal{F} \mbox{ of 
	elements } K \in \mathcal{T}_h \, \big| \mathcal{F} \subset D \right\}, \\
	\mathcal{E}^B_h &:= \left\{\mbox{ Set of faces } \mathcal{F} \mbox{ of 
	elements } K \in\mathcal{T}_h \, \big| \mathcal{F} \subset \pa D \right\}.
\end{align*}

Let $\bH^1(\mathcal{T}_h)$ denote the vector-valued broken Sobolev space over 
the partition $\mathcal{T}_h$ defined as
\begin{align*}
	\bH^1(\mathcal{T}_h) := \left\{ \bfv \in \bL^2(D) \, \big| \, \bfv \in 
	\bH^1(K) \mbox{ for all } K \in \mathcal{T}_h  \right\}.
\end{align*}
Let $\bfV_h$ denote the space of piecewise linear functions over the partition 
$\mathcal{T}_h$ defined as
\begin{align*}
	\bfV_h := \left\{ \bfv \in \bH^1(\mathcal{T}_h) \, \big| \, \bfv \in 
	\bfP_1(K) \mbox{ for all } K \in \mathcal{T}_h  \right\},
\end{align*}
where $\bfP_1(K)$ is the space of vector-valued piecewise linear functions on 
$K \in \mathcal{T}_h$. 
 
The finite dimensional function space $\bfV_h$ includes functions that are 
discontinuous at the interior edges of all elements of the partition $\mathcal{T}_h$.  
To deal with these discontinuities, it is standard practice to define jump and 
average operators at these element boundaries.  Let $[\bfv]$ and $\{ \bfv \}$ 
denote the jump and average operator of $\bfv$ on a face $\mathcal{F} \in 
\mathcal{E}_h := \mathcal{E}^I_h \cup \mathcal{E}^B_h$.  For any $\mathcal{F} 
\in \mathcal{E}^I_h$ there exists $K,K' \in \mathcal{T}_h$ such that $
\mathcal{F} = \pa K \cap \pa K'$.  Thus, for $\mathcal{F} \in \mathcal{E}^I_h$ 
define
\begin{align*}
	[\bfv]|_{\mathcal{F}} &:= \left\{
		\begin{array}{l r}
			\bfv_K - \bfv_{K'}, & \mbox{ if the global label of } K \mbox{ is 
			larger than } K' \\
			\bfv_{K'} - \bfv_{K}, & \mbox{ if the global label of } K' \mbox{ 
			is larger than } K
		\end{array}
	\right. , \\
	\{ \bfv \}|_{\mathcal{F}} &:= \frac{1}{2} \big( \bfv_K + \bfv_{K'} \big).
\end{align*}
For $\mathcal{F} \in \mathcal{E}^B_h$, we set $[\bfv]|_\mathcal{F} = \{\bfv\}|_
\mathcal{F} = \bfv|_\mathcal{F}$. For $\mathcal{F} \in \mathcal{E}^I_h$ such 
that $\mathcal{F} = \pa K \cap \pa K'$ we define $\bfnu_\mathcal{F}$ as the 
unit normal to the face $\mathcal{F}$ which points outward of $K$ where 
$K$ has larger global label than $K'$. For $\mathcal{F} \in \mathcal{E}^B_h$ we 
define $\bfnu_\mathcal{F} = \bfnu$ the unit normal vector to the face $
\mathcal{F}$ which points outward of the domain $D$.

\bigskip

\subsection{IP-DG method for the deterministic time-harmonic Maxwell problem}
In this section we introduce the IP-DG method of \cite{Feng_Wu_14} that will be 
used in our overall MCIP-DG procedure. We consider the following 
deterministic time-harmonic Maxwell's equations:
\begin{alignat}{2}
	\curl \curl \bfE - k^2 \bfE &= \bfF &&\qquad \mbox{in } D, 
	\label{Eq:DeterministicPDE1}\\
	 \curl \bfE \times \bfnu - \i k \lambda \bfE_T &= \mathbf{0} &&\qquad \mbox{on } 
	 \pa D.
	 \label{Eq:DeterministicPDE2}
\end{alignat}
The IP-DG approximation $\bfE^h \in \bfV_h$ of the solution $\bfE$ is defined by 
seeking $\bfE^h \in \bfV_h$ that satisfies
\begin{align} \label{Eq:IP-DG}
	a_h(\bfE^h,\bfv^h) = (\bfF, \bfv^h)_D \qquad \forall \bfv^h \in \bfV_h,
\end{align}
where
\begin{align*}
	a_h(\bfE^h, \bfv^h) &:= b_h(\bfE^h, \bfv^h) - k^2(\bfE^h,\bfv^h)_D - \i 
	k \lambda \langle \bfE^h_T, \bfv^h_T \rangle_{\pa D}, \\
	b_h(\bfE^h, \bfv^h) &:= \sum_{K \in \mathcal{T}_h} ( \curl \bfE^h, \curl 
	\bfv^h)_K \\
	& \qquad - \sum_{\mathcal{F} \in \mathcal{E}^I_h} \Big( \big \langle \{ 
	\curl \bfE^h \times \bfnu_\mathcal{F} \}, [\bfv^h_T] \big 
	\rangle_{\mathcal{F}} + \big \langle [\bfE^h_T], \{ \curl \bfv^h \times 
	\bfnu_\mathcal{F} \} \big \rangle_{\mathcal{F}} \Big) \\
	& \qquad - \i \Big( J_0(\bfE^h, \bfv^h) + J_1(\bfE^h,\bfv^h) \Big), \\
	J_0(\bfE^h, \bfv^h) &:= \sum_{\mathcal{F} \in \mathcal{E}^I_h} 
	\frac{\gamma_{0}}{h} \big \langle [\bfE^h_T], 
	[\bfv^h_T] \big \rangle_\mathcal{F}, \\
	J_1(\bfE^h, \bfv^h) &:= \sum_{\mathcal{F} \in \mathcal{E}^I_h} \gamma_{1}h \big \langle [\curl \bfE^h \times \bfnu_
	\mathcal{F}], [\curl \bfv^h \times \bfnu_\mathcal{F}] \big \rangle_
	\mathcal{F}.
\end{align*}
Here $\gamma_{0}$ and $\gamma_{1}$ denote nonnegative 
penalty parameters that will be specified later. To aid in the analysis of this 
IP-DG method we define the following norms and semi-norms on $\bfV_h$:
\begin{align*}
	\| \bfv \|^2_{L^2(\mathcal{T}_h)} &:= \sum_{K \in \mathcal{T}_h} \| \bfv 
	\|^2_{L^2(K)}, \\
	| \bfv |^2_{DG} &:= \| \curl \bfv \|^2_{L^2(\mathcal{T}_h)} + J_0(\bfv,
	\bfv) + J_1(\bfv,\bfv), \\
	\| \bfv \|^2_{DG} &:= | \bfv |^2_{DG} + \| \bfv \|^2_{L^2(D)}.
\end{align*}

The next theorem comes from \cite{Feng_Wu_14} and demonstrates the 
unconditional stability of the above IP-DG method.  This theorem is split 
into two parts based on the size of the spatial partition parameter $h$ that is chosen.  
The first set of estimates are valid for any size of $h > 0$ while the second 
set of estimates are valid when $k^3h^2 = O(1)$ (considered the asymptotic mesh regime) 
and are sharper estimates.

\begin{theorem} \label{Thm:IP-DGStability}
	Let $\bfE^h \in \bfV_h$ be a solution to \eqref{Eq:IP-DG}.  
\begin{enumerate}[{\rm i)}]
\item For any $k,\lambda,h,\gamma_{0},\gamma_{1} > 0$ the following 
estimates hold:
\begin{align*}
	\| \curl \bfE^h \|_{L^2(\mathcal{T}_h)} + k \| \bfE^h \|_{L^2(D)} &\leq C 
	k^{-1} \csta \| \bfF \|_{L^2(D)}, \\
	\left( J_0(\bfE^h,\bfE^h) + J_1(\bfE^h,\bfE^h) + k \lambda \| \bfE^h_T \|
	^2_{L^2(\pa D)} \right)^{\frac{1}{2}} &\leq C k^{-1} \csta^{\frac{1}{2}} \| 
	\bfF \|_{L^2(D)},
\end{align*}
where $C$ is a constant independent of $h$, $\gamma_0$, $\gamma_1$, $
\lambda$, and $k$, and 
\begin{align*}
	\csta := \frac{1}{k \lambda h} + \frac{1}{\gamma_1 k^2 h^2} 
	+ \frac{1}{\gamma_0}+1.
\end{align*}
\item For $h$ in the asymptotic regime, $k^3h^2 = O(1)$, the following 
estimate holds:
\begin{align*}
	\| \bfE^h \|_{L^2(D)} + \| \bfE^h \|_{L^2(\pa D)} + \frac{1}{k}| \bfE^h |
	_{DG} \leq \widehat{C}_0 \left( \frac{1}{k} + \frac{1}{k^2} \right) 
	\| \bfF \|_{L^2(D)}.
\end{align*}
\end{enumerate}
\end{theorem}

Unique solvability to the IP-DG equation \eqref{Eq:IP-DG} is an immediate 
consequence of the above unconditional stability result.
\begin{corollary}
	There exists a unique solution to \eqref{Eq:IP-DG} for any fixed set of 
	parameters $k,\lambda,h,\gamma_{0},\gamma_{1} > 0$.
\end{corollary}

The following theorem gives error estimates for this IP-DG approximation.
\begin{theorem} \label{Thm:IPDGError}
	Let $\bfE \in \bH^2(D)$ solve 
	\eqref{Eq:DeterministicPDE1}--\eqref{Eq:DeterministicPDE2}
	 and $\bfE^h \in \bfV_h$ solve \eqref{Eq:IP-DG}.
	 \begin{enumerate}[i.)]
	 	\item For any $k,\lambda,h,\gamma_{0},\gamma_{1} > 0$ the 
	 	following estimates hold:
	 	\begin{align*}
	 		\| \bfE - \bfE^h \|_{DG} &\leq C \left( h + \hatcsta(1+\gamma_1)
	 		\left(k^2 h^2 + k \lambda h^{\frac{3}{2}}\right) \right) 
	 		\mathcal{R}
	 		(\bfE), \\
	 		\| \bfE - \bfE^h \|_{L^2(D)} &\leq C \left( h^2 + \hatcsta k^{-1} 
	 		\left( k^2 h^2 + k \lambda h^{\frac{3}{2}} \right) \right) ( 1 + 
	 		\gamma_1) \mathcal{R}(\bfE),
	 	\end{align*}
	 	where
	 	\begin{align*}
	 		\mathcal{R}(\bfE) &:= (1 + \gamma_1)^{\frac{1}{2}} \left( \| \bfE 
	 		\|_{H^1(D)}^2 + \| \curl \bfE \|^2_{H^1(D)} \right)^{\frac{1}{2}} 
	 		+ \| \bfE \|_{H^2(D)}, \\
	 		\hatcsta &:= \max \left\{ k^{-1}(1+\csta), \big( k^{-1}
	 		\lambda^{-1} (1+\csta) \big)^{\frac{1}{2}} \right\}.
	 	\end{align*}
	 	\item In the case that $k^3 h^2 = O(1)$, the following estimates hold:
	 	\begin{align*}
	 		\| \bfE - \bfE^h \|_{DG} &\leq h C_1 \mathcal{R}(\bfE), \\
	 		\| \bfE - \bfE^h \|_{L^2(D)} &\leq h^2 C_2\mathcal{R}(\bfE).
	 	\end{align*}
	 \end{enumerate}
\end{theorem}
The proof for part i) can be found in \cite{Feng_Wu_14} and for part ii) in \cite{Monk_03}.  
It is well known that the solution $\bfE$ to the deterministic time-harmonic Maxwell's equations 
may not belong to $\bH^2(D)$ in some cases (c.f. \cite{Hiptmair_Moiola_Perugia_13}).  
On the other hand, if $\bfE \in \bH^2(D)$ then it can be shown that
\begin{align} \label{Eq:H2_Estimate}
	\mathcal{R} ( \bfE ) \leq C (1 + k) \| \bff \|_{L^2(D)}.
\end{align}
To keep estimates tractable for the rest of the paper,
we will assume that $k^3h^2 = O(1)$ and \eqref{Eq:H2_Estimate} holds.

We also note that $\E(\bfE_n)$ for $n \geq 1$ cannot be computed directly 
using the IP-DG method described in this section, due to the multiplicative 
structure of the right-hand side of \eqref{Eq:ModePDE3}. Thus, we will need one 
more layer of approximations to achieve a fully computable solution.  The next 
section will add a Monte Carlo method to the IP-DG method of this section to 
obtain an MCIP-DG method for approximating $\E(\bfE_n)$. However, 
one can approximate $\E(\bfE_0)$ directly by prescribing the   
function $\bfF = \E(\bff)$.  This is due to the linear 
nature of the expectation operator $\E(\cdot)$.  

\subsection{MCIP-DG method for approximating $\mathbf{\E(\bfE_n)}$ for $\mathbf{n \geq 0}$}
In this subsection we present our MCIP-DG method for approximating the expectation $\E(\bfE_n)$
of each mode function $\bfE_n$. 
Although it was noted earlier that the IP-DG method is valid in the pre-asymptotic mesh 
regime, we will only carry out the error analysis in the asymptotic mesh 
regime, $k^3h^2 = O(1)$, to avoid technicalities.  The estimates obtained in this subsection 
are similar to those of \cite{Feng_Lin_Lorton_15,Feng_Lorton_17} for the random Helmholtz 
problem and random elastic Helmholtz problem. 

Recall that each mode function $\bfE_n$ satisfies the following ``nearly 
deterministic" time-harmonic Maxwell's equations:
\begin{align*}
	& \curl \curl \bfE_n - k^2 \bfE_n = \bfS_n, \\
	& \curl \bfE_n \times \bfnu - \i k \lambda \bfE_n = \mathbf{0},
\end{align*}
where
\begin{align*}
	&\bfS_0 := \bff, \qquad \bfE_{-1} := \mathbf{0},\qquad \bfS_n := 2k^2 \eta \bfE_{n-1} + k^2\eta^2 
	\bfE_{n-2} \quad\forall n \geq 1.
\end{align*} 
Following the standard Monte Carlo procedure (c.f. 
\cite{Babuska_Tempone_Zouraris_04}) we let $M$ be a large integer indicating 
the number of samples to be taken to generate the Monte Carlo approximation.  
For each $j = 1, 2, \dots, M$ we obtain i.i.d. realizations of the source term 
$\bff(\om_j,\cdot) \in \bL^2(D)$ and 
the random coefficient $\eta(\om_j,\cdot) \in W_C^{1,\infty}(D)$.  With each 
realization of the data a sample mode function $\bfE_n^h(\om_j,\cdot) \in 
\bfV_h$ is found 
by solving the following IP-DG equation
\begin{align} \label{Eq:MCIP-DG}
	a_h(\bfE_n^h(\om_j,\cdot),\bfv^h) = (\bfS_n^h(\om_j,\cdot), \bfv^h)_D  
	\qquad \forall \bfv^h \in \bfV_h,
\end{align}
where
\begin{align*}
	\bfS^h_0(\om_j,\cdot) &:= \bff(\om_j,\cdot), \qquad \bfE^h_{-1} := \mathbf{0},\\
     \bfS^h_n(\om_j,\cdot) 
	&:= 2k^2 \eta \bfE_{n-1}^h(\om_j,\cdot) + k^2\eta^2 
	\bfE_{n-2}^h(\om_j,\cdot) \quad \forall n \geq 1.
\end{align*}
Due to the definition of $\bfS^h_n$, each realization $
\bfE^h_n(\om_j,\cdot)$ must be computed recursively. It is also important to 
note that for the sake of computability the source term $\bfS_n$
must be replaced with a discrete versions $\bfS^h_n$.  This adds 
another source of error that is dealt with in the overall error analysis.  The 
following theorem gives stability estimates for the discretized mode functions 
$\bfE^h_n$.  These stability estimates are a result of the stability estimates in
part ii) of Theorem \ref{Thm:IP-DGStability}. along with an induction argument 
like the one used to prove estimates in Theorem \ref{Thm:MultiModes}. 
\begin{lemma} \label{Lem:DiscreteModes}
	Let $\bff \in \bL^2\big(\Om,\bL^2(D))\big)$. Then for each $n \geq 0$, 
	there exists a 
	unique solution $\bfE^h_n \in \bL^2(\Om,\bfV^h)$ satisfying 
	\eqref{Eq:MCIP-DG}.  Moreover, for $h$ chosen such that $k^3h^2 = O(1)$, $
	\bfE^h_n$ satisfies the following stability estimates:
	\begin{align}
		\E \big( \| \bfE^h_n \|^2_{L^2(D)} \big) + \E \big( \| \bfE^h_n \|
		^2_{L^2(\pa D)} \big)
		\label{Eq:DiscreteModeEst1}
		& \leq \widehat{C}(n,k) \left(\frac{1}{k} + 
		\frac{1}{k^2} \right)^2 \E \big( \| \bff \|^2_{L^2(D)} \big), \\
		\label{Eq:DiscreteModeEst2}
		\E \big( \| \bfE^h_n \|^2_{DG} \big) & \leq \widehat{C}(n,k) \left(1 + 
		\frac{1}{k} \right)^2 \E \big( \| \bff \|^2_{L^2(D)} \big),
	\end{align}
	where
	\begin{align*}
		\widehat{C}(0,k) := \widehat{C}_0, \qquad \widehat{C}(n,k) := 4^{2n-1}
		\widehat{C}_0^{2n+2}(1+k)^{2n} \qquad\forall n \geq 1.
	\end{align*}
\end{lemma}

Now we define our MCIP-DG approximation $\bfphi^h_n$ of $\E(\bfE_n)$ 
to be the following statistical average:
\begin{align} \label{Eq:MCIP-DGMode}
	\bfphi^h_n := \sum_{j = 1}^M \bfE^h_n(\om_j,\cdot).
\end{align}
To analyze the error between $\bfphi^h_n$ and $\E(\bfE_n)$ we note that the 
error can be decomposed as
\begin{align*}
	\E(\bfE_n) - \bfphi^h_n  = \Big( \E(\bfE_n) - \E(\bfE_n^h) \Big) + \Big( 
	\E(\bfE^h_n) -  \bfphi^h_n \Big),
\end{align*}
i.e. the error associated to approximating each mode function with its IP-DG 
approximation and the error associated to approximating the expectation with a 
statistical average generated from the Monte Carlo method.

To prove error estimates between the mode function $\bfE_n$ and the IP-DG approximation to the mode function $\bfE^h_n$ we will make use of the following lemma.
\begin{lemma}\label{Lem:Trivial_Inequality}
Let $\gamma,\beta > 0$ be two real numbers, $\{c_n\}_{n\geq 0}$ and
$\{\alpha_n\}_{n\geq 0}$ be two sequences of nonnegative numbers such that
\begin{equation}
c_0\leq \gamma\alpha_0, \quad
c_n\leq \beta c_{n-1} +\gamma \alpha_n \qquad\forall n\geq 1.
\end{equation}
Then there holds
\begin{equation}
c_n\leq \gamma \sum_{j=0}^n \beta^{n-j} \alpha_j \qquad\forall n\geq 1.
\end{equation}
\end{lemma}
The proof is trivial so it is not given here.
We will also use a decomposition of the form
\begin{align}\label{Eq:IP-DG_Mode_Decomp_1}
	\bfE_n - \bfE^h_n = \big( \bfE_n - \widetilde{\bfE}^h_n \big) + \big(\widetilde{\bfE}^h_n - \bfE^h_n \big),
\end{align}
where $\widetilde{\bfE}^h_n(\om_j,\cdot) \in \bfV_h$ solves the following IP-DG equation
\begin{align} \label{Eq:IP-DG_Mode_Decomp_2}
	a_h(\bfE_n^h(\om_j,\cdot),\bfv^h) = (\bfS_n(\om_j,\cdot), \bfv^h)_D  
	\qquad \forall \bfv^h \in \bfV_h.
\end{align}
With these tools in hand we now give the following theorem which characterizes the error associated to approximating $\bfE_n$ with an IP-DG approximation $\bfE^h_n$.

\begin{theorem} \label{Thm:MultiModesIP-DGError}
	Suppose that $k^3h^2 = O(1)$ and $\bfE_n \in \bL^2(\Om,\bH^2(D))$, then 
	the following estimates hold
	\begin{align}
		\E \big( \| \bfE_n - \bfE^h_n \|_{L^2(D)} \big) &\leq \widetilde{C}_0 
		(1+k) h^2 \sum_{j = 0}^n \big[ \widehat{C}_0 (2k + 2) \big]^{n-j} \E 
		\big(\| \bfS_j \|_{L^2(D)} \big), \label{Eq:MultiModesIP-DGError1} \\
		\E \big( \| \bfE_n - \bfE^h_n \|_{DG} \big) & \leq C \widetilde{C}_0 (1+k)h 
		\sum_{j = 0}^n \big[ \widehat{C}_0 (2k + 2) \big]^{n-j} \E 
		\big(\| \bfS_j \|_{L^2(D)} \big), \label{Eq:MultiModesIP-DGError2}
	\end{align}
	where $C$, $\widetilde{C}_0$, and $\widehat{C}_0$ are constants 
	independent of $k$ and $h$.
\end{theorem}

\begin{proof}
	As an immediate consequence of Theorem \ref{Thm:IPDGError} and 
	\eqref{Eq:H2_Estimate} the following estimates on the error $\bfE_n - 
	\bftE_n^h$ hold
	\begin{align} \label{Eq:MultiModesIP-DGError3}
		\E \big( \| \bfE_n - \bftE_n^h \|_{DG} \big) &\leq C h (1 + k) 
		\E \big( \| \bfS_n \|_{L^2(D)} \big), \\
		\E \big( \| \bfE_n - \bftE_n^h \|_{L^2(D)} \big) &\leq C h^2 (1 + k) 
		\E \big( \| \bfS_n \|_{L^2(D)} \big). \label{Eq:MultiModesIP-DGError4}
	\end{align}
	To estimate the error $\bftE^h_n - \bfE^h_n$ we subtract 
	\eqref{Eq:MCIP-DG} from \eqref{Eq:IP-DG_Mode_Decomp_2} to obtain
	\begin{align*}
	 a_h\big(\bftE^h_n(\om_j,\cdot) - \bfE^h_n(\om_j,\cdot),\bfv^h \big) = 
	 \big(\bfS_n(\om_j,\cdot) - \bfS^h_n(\om_j,\cdot),\bfv^h \big)_{D} \qquad 
	 \forall \bfv^h \in \bfV_h.
	\end{align*}
	Thus, $\bfE_n(\om_j,\cdot) - \bftE_n^h(\om_j,\cdot)$ solves 
	\eqref{Eq:IP-DG} with $\bfF = \bfS_n(\om_j,\cdot) - \bfS^h_n(\om_j,\cdot)$
	and we can apply part ii) of Theorem \ref{Thm:IP-DGStability} to obtain
	\begin{align}\label{Eq:MultiModesIP-DGError5}
		&\E \big( k \| \bftE^h_n - \bfE^h_n \|_{L^2(D)} + k \| \bftE^h_n -
		 \bfE^h_n \|_{L^2(\pa D)} + | \bftE^h_n -
		  \bfE^h_n  |_{DG} \big) \\
		& \qquad \leq \widehat{C}_0 \left(1 + \frac{1}{k} \right) 
		\E \big( \|\bfS_n - \bfS^h_n \|_{L^2(D)} \big) \notag \\
		& \qquad \leq \widehat{C}_0 k(k+1) \E \big( 2 \| \bfE_{n-1} - 
		\bfE^h_{n-1} \|_{L^2(D)} + \| \bfE_{n-2} - \bfE^h_{n-2} \|_{L^2(D)} 
		\big). \notag
	\end{align}
	Combining \eqref{Eq:MultiModesIP-DGError4} and 
	\eqref{Eq:MultiModesIP-DGError5} yields
	\begin{align} \label{Eq:MultiModesIP-DGError6}
		&\E \big( \| \bfE_n - \bfE^h_n \|_{L^2(D)} + 
		\| \bfE_n - \bfE^h_n \|_{L^2(\pa D)} \big) \\
		& \qquad \leq \widehat{C}_0 (k+1) \E \big( 2\| \bfE_{n-1} - 
		\bfE^h_{n-1} \|_{L^2(D)} + \| \bfE_{n-2} - \bfE^h_{n-2} \|_{L^2(D)} 
		\big) \notag \\
		& \qquad \qquad + \widetilde{C}_0 (k+1) h^2 
		\E \big( \| \bfS_n \|_{L^2(D)} \big). \notag
	\end{align}
	We then apply an inverse inequality along with 
	\eqref{Eq:MultiModesIP-DGError5} and \eqref{Eq:MultiModesIP-DGError3} to 
	find
	\begin{align} \label{Eq:MultiModesIP-DGError7}
		&\E \big( \| \bfE_n - \bfE^h_n \|_{DG} \big) \\
		& \qquad \leq \E \big( \| \bftE^h_n - \bfE^h_n \|_{DG} + 
		 \| \bfE_n -\bftE^h_n \|_{DG} \big) \notag \\
		& \qquad \leq C h^{-1} \E \big( \| \bftE^h_n - \bfE^h_n \|_{L^2(D)}  
		\big) + \E \big( \| \bfE_n -\bftE^h_n \|_{DG} \big) \notag \\
		& \qquad \leq C \widehat{C}_0 h^{-1} (k+1) \E \big( 2\| \bfE_{n-1} - 
		\bfE^h_{n-1} \|_{L^2(D)} + \| \bfE_{n-2} - \bfE^h_{n-2} \|_{L^2(D)} 
		\big) \notag \\
		& \qquad \qquad + \widetilde{C}_0 (k+1) h 
		\E \big( \| \bfS_n \|_{L^2(D)} \big). \notag
	\end{align}
	
	We note that \eqref{Eq:MultiModesIP-DGError6} and
	 \eqref{Eq:MultiModesIP-DGError7} define recursive estimates for the error 
	 $\bfE_n - \bfE^h_n$. By definition
	\begin{align*}
		\bfE_{-2} = \bfE_{-1} = \bfE^h_{-2} = \bfE^h_{-1} = \mathbf{0},
	\end{align*}
	we get
	\begin{align*}
		\E \big( \| \bfE_0 - \bfE^h_0 \|_{L^2(D)} &\leq \widetilde{C}_0 (k+1) 
		h^2 \E \big( \| \bfS_0 \|_{L^2(D)} \big), \\
		\E \big( \| \bfE_0 - \bfE^h_0 \|_{DG} \big) &\leq \widetilde{C}_0 
		(k+1) h \E \big( \| \bfS_0 \|_{L^2(D)} \big).
	\end{align*}
	Define
	\begin{align*}
		c_n &: = \E \big( \| \bfE_{n} - \bfE^h_{n} \|_{L^2(D)} +
		 \| \bfE_{n-1} - \bfE^h_{n-1} \|_{L^2(D)}  \big), \\
		\beta &:= \widehat{C}_0 (2k + 2), \qquad \gamma :=  
		\widetilde{C}_0 (k+1)h^2, \qquad \alpha_n := 
		\E \big( \| \bfS_n \|_{L^2(D)} \big),
	\end{align*}
	we now apply Lemma \ref{Lem:Trivial_Inequality} with these choices to obtain 
	\eqref{Eq:MultiModesIP-DGError1}. Finally,  combining  
	\eqref{Eq:MultiModesIP-DGError7} and \eqref{Eq:MultiModesIP-DGError1} gives
	\eqref{Eq:MultiModesIP-DGError2}.
\end{proof}

To characterize the error associated with approximating the expected value by
the Monte Carlo method we make use of the 
following well known lemma (c.f. \cite{Babuska_Tempone_Zouraris_04,Liu_Riviere_13}).
\begin{lemma} \label{Lem:MultiModesMCError}
For $n \geq 0$ the following estimates hold
\begin{align}
	\E \big( \| \E(\bfE_n^h) - \bfphi^h_n \|^2_{L^2(D)} \big)& \leq 
	\frac{1}{M}\E \big( \| \bfE^h_n \|^2_{L^2(D)} \big), 
	\label{MultiModesMCError1} \\
	\E \big( \| \E(\bfE_n^h) - \bfphi^h_n \|^2_{DG} \big)& \leq 
	\frac{1}{M}\E \big( \| \bfE^h_n \|^2_{DG} \big).
	\label{MultiModesMCError2}
\end{align}
\end{lemma}
From Lemma \ref{Lem:DiscreteModes} and \ref{Lem:MultiModesMCError} we obtain the following theorem.
\begin{theorem} \label{Thm:MultiModesMCError}
	Suppose that $k^3h^2 = O(1)$, then the following estimates hold:
	\begin{align}
		\E \big( \| \E(\bfE_n^h) - \bfphi^h_n \|^2_{L^2(D)} \big)& \leq
		\frac{1}{M}\widehat{C}(n,k) \left(\frac{1}{k} + 
		\frac{1}{k^2} \right)^2 \E \big( \| \bff \|^2_{L^2(D)} \big), 
		\label{Eq:MultiModesMCError1} \\
		\E \big( \| \E(\bfE_n^h) - \bfphi^h_n \|^2_{DG} \big)& \leq
		\frac{1}{M}\widehat{C}(n,k) \left(1 + 
		\frac{1}{k} \right)^2 \E \big( \| \bff \|^2_{L^2(D)} \big). 
		\label{Eq:MultiModesMCError2}
	\end{align}
\end{theorem}
As expected, the error associated with approximating $\E(\bfE_n^h)$ using the 
Monte Carlo method is on the order $O(M^{\frac{1}{2}})$.  Thus, to ensure 
convergence $M$ must be taken to be sufficiently large.

\section{The overall numerical procedure} \label{sec:Numerical_Procedure}
In this section, we will present the overall numerical 
algorithm based on a combination of the multi-modes expansion 
of the solution given in \eqref{Eq:MultiModes} and the MCIP-DG method for computing each mode.
An acceleration strategy is proposed to obtain the mode functions so that the whole algorithm can be 
implemented in an efficient way. It should be pointed out that, instead of employing the Monte Carlo method for sampling, more efficient sampling techniques such as 
quasi-Monte Carlo methods or stochastic collocation methods can also be applied to compute the expectation, we omit the discussion here for conciseness.

\subsection{The numerical algorithm, linear solver, and computational complexity}
Before describing the multi-modes MCIP-DG method, we begin this section by 
giving the ``standard" MCIP-DG method for obtaining $\E(\bfE^\veps)$.  We use 
the term standard MCIP-DG method to describe a Monte Carlo interior penalty 
discontinuous Galerkin approximation which does not make use of the 
multi-modes expansion of the solution. The reason for introducing the standard 
MCIP-DG method is twofold.  First, we use this method as a test-stone for 
comparison with our multi-modes MCIP-DG method.  In particular, we will show 
that the multi-modes MCIP-DG method is far superior to the standard MCIP-DG 
method in terms of computational time needed for completion of the algorithm.  
Second, due to the difficulty in obtaining a closed-form solution for 
\eqref{Eq:PDE1} --\eqref{Eq:PDE2} we will compare the solution from the 
multi-modes MCIP-DG method to the solution obtained using the standard MCIP-DG 
method in all of our numerical tests later in the paper.  We do this because 
the standard method is known to converge to the true solution.

To describe the standard MCIP-DG method we start by defining the following 
IP-DG sesquilinear form:
\begin{align*}
	\hat{a}_{h,j}(\bfE^h, \bfv^h) &:= b_h(\bfE^h, \bfv^h) - k^2 
	\big(\alpha^2(\om_j,\cdot)\bfE^h,\bfv^h\big)_D - \i 
	k \lambda \langle \bfE^h_T, \bfv^h_T \rangle_{\pa D},
\end{align*}
where $\alpha(\om_j,\cdot)$ indicates a given realization of the random 
coefficient in \eqref{Eq:PDE1}.  Note that this sesquilinear form is similar 
in nature to that from \eqref{Eq:MCIP-DG} with the exception that $
\alpha(\om_j,\cdot)$ is now present in the second term of this sesquilinear 
form.  With this sesquilinear form we define the ``standard" MCIP-DG method.\\

\noindent {\bf Algorithm 1 (Standard MCIP-DG)}

\begin{description}
\item Input $\bff, \eta, \veps, k, h, M.$
\item Set $\bftPsi^\veps_h(\cdot)= \mathbf{0}$ (initializing).
\begin{description}
\item For $j=1,2,\cdots, M$
\item Obtain realizations $\bff(\om_j,\cdot)$ and $
\eta(\om_j,\cdot)$.
\item Solve for $\hat{\bfE}^h(\omega_j,\cdot) \in \bV^h$ such that
\[
\hat{a}_{h,j}\big( \hat{\bfE}^h(\omega_j,\cdot), \bfv_h \big) = 
\big(\bff(\omega_j,\cdot), \bfv_h\big)_D  \qquad\forall v_h\in \bV^h.
\]
\item Set $\bftPsi^\veps_h(\cdot) \leftarrow \bftPsi^\veps_h(\cdot) +\frac{1}
{M} \hat{\bfE}^h(\omega_j,\cdot)$.
\item Endfor
\end{description}
\item Output $\bftPsi^\veps_h(\cdot)$.
\end{description}
For the rest of the paper we will use the notation $\bftPsi^\veps_h(\cdot)$ to 
denote the ``standard" MCIP-DG approximation of $\E(\bfE^\veps)$.

We now are ready to state our multi-modes MCIP-DG algorithm. The key to obtaining an 
efficient algorithm 
is to leverage the fact that all the mode functions $\bfE_n$ satisfy a random 
PDE with the same left-hand side operators.  The random coefficients and 
source terms only show up on the right-hand side in a recursive fashion (see 
\eqref{Eq:ModePDE1}--\eqref{Eq:ModePDE4}).  This means that each IP-DG 
approximation of $\bfE_n$ will generate the same stiffness matrix $A$, 
regardless of the sample $\om_j$.
With this in mind, this matrix needs to be computed only once in the entire Monte Carlo approximation and thus, 
only one $LU$ decomposition needs to be performed for the entire algorithm.  
Then for each realization of coefficient and source data, the stored $LU$ 
decomposition along with forward and backward substitution can be used to 
compute the corresponding solution.

The algorithm based on the multi-modes MCIP-DG method is described 
below.\\

\noindent
{\bf Algorithm 2 (Multi-Modes MCIP-DG)}
\begin{description}
\itemsep0.1cm
\item Input $\bff, \eta, \veps, k, h, M,N$
\item Set $\bfPsi^\veps_{h,N}(\cdot)=0$ (initializing).
\item Generate the stiffness matrix $A$ from the sesquilinear form 
$a_h(\cdot,\cdot)$ on $\bV^h \times \bV^h$.
\item Compute and store the $LU$ decomposition of $A$.
\begin{description}
\itemsep0.1cm
\item For $j=1,2,\cdots, M$
\item Obtain realizations $\bff(\om_j,\cdot)$ and $
\eta(\om_j,\cdot)$.
\item Set $\bfS^h_0(\omega_j,\cdot)=\bff(\omega_j,\cdot)$.
\item Set $\bfE^h_{-1}(\omega_j,\cdot)= \mathbf{0}$.
\item Set $\bfE^\veps_{h,N}(\omega_j,\cdot)=\mathbf{0}$ (initializing).
\begin{description}
\itemsep0.1cm
\item For $n=0,1,\cdots, N-1$
\item Solve for $\bfE^h_n(\omega_j,\cdot) \in \bV^h$ such that
\[
a_h\bigl( \bfE^h_n(\omega_j,\cdot), \bfv_h \bigr) = \bigl(\bfS^h_n(\omega_j,
\cdot), \bfv_h\bigr)_D \qquad\forall \bfv_h\in \bV^h,
\]
using forward and backward substitution.
\item Set $\bfE^\veps_{h,N}(\omega_j,\cdot)\leftarrow \bfE^\veps_{h,N}
(\omega_j,\cdot) +\veps^n 
\bfE^h_n(\omega_j,\cdot)$.
\item Set $\bfS^h_{n+1}(\omega_j,\cdot)=2k^2 \eta(\omega_j,\cdot) 
\bfE^h_n(\omega_j,\cdot)
+ k^2 \eta(\omega_j,\cdot)^2 \bfE^h_{n-1}(\omega_j,\cdot)$.
\item Endfor
\end{description}
\item Set $\bfPsi^\veps_{h,N}(\cdot) \leftarrow \bfPsi^\veps_{h,N}(\cdot) +
\frac{1}{M} 
\bfE^\veps_{h,N}(\omega_j,\cdot)$.
\item Endfor
\end{description}
\item Output $\bfPsi^\veps_{h,N}(\cdot)$.
\end{description}
In the rest of the paper $\bfPsi^\veps_{h,N}$ is used to denote the multi-modes MCIP-DG 
approximation to $\E(\bfE^\veps)$ calculated by using Algorithm 2. Though $\bfphi^h_n$ as 
defined in \eqref{Eq:MCIP-DGMode} does not show up 
explicitly in Algorithm 2, we note that
\begin{align*}
	\bfPsi^\veps_{h,N} = \sum_{n=0}^{N-1} \veps^n \bfphi^h_n.
\end{align*}
This identity will be used in the convergence analysis given in the next 
section.

To demonstrate the efficiency of Algorithm 2, let $L = \frac{1}{h}$ where $h$ 
is the spatial mesh size used in the IP-DG method (see Section 
\ref{sec:MCIP-DG}), and we note for convergence of the IP-DG method we must 
choose $h$ to be small ensuring the parameter $L$ is large.  As stated in 
\cite{Feng_Lin_Lorton_15,Feng_Lin_Lorton_16,Feng_Lorton_17}, Algorithm 1 
requires $O(ML^{9})$ multiplications versus $O(L^{9} + MNL^{6})$ number of 
multiplications used in Algorithm 2.  In practice, the number of modes $N$ is
relatively small (see Theorem 
\ref{Thm:Convergence}), thus we can treat this parameter as a constant. To 
ensure the error associated with the IP-DG method, measured in the $L^2$ norm, is 
of the same order as the error due to the Monte Carlo simulation, we set 
$M = L^4$.  With this choice the number of multiplications used in 
Algorithm 1 is $O(L^{13})$ versus the $O(L^{9} + L^{10})$ used in 
Algorithm 2. Thus, big savings in the computational cost by using the 
multi-modes MCIP-DG method in Algorithm 2 over using the standard MCIP-DG 
method in Algorithm 1 is achieved.

Moreover, the proposed MCIP-DG method maintains the parallelism feature
of the standard Monte Carlo method because the outer 
loop of Algorithm 2 can be run simultaneously. Hence, it allows 
implementing the algorithm in parallel, thus resulting in additional speedup 
of the algorithm.

\subsection{Convergence analysis} \label{subsel:convergence}
In this section, we will give the convergence analysis for 
the proposed multi-modes MCIP-DG method.  To do so, we first decompose the total
error as follows: 
\begin{align} \label{Eq:Error_Decomposition}
	\E(\bfE^\veps)-\bfPsi^\veps_{h,N} 
	&=\bigl(\E(\bfE^\veps)-\E(\bfE^\veps_N)\bigr) + 
	\bigl( \E(\bfE^\veps_N) - \E(\bfE^\veps_{h,N})\bigr) \\
	&\qquad +\bigl( \E(\bfE^
	\veps_{h,N})-\bfPsi^\veps_{h,N} 
	\bigr). \notag
\end{align}
Here, the first term in the decomposition represents the error associated 
with the finite-modes representation of the solution, the second term is the error 
contributed by the IP-DG method, and the last term is the error arising in the sampling by
the Monte Carlo method.  We note that the error associated with 
approximating $\bfE^\veps$ by a finite-modes representation was already 
presented in Theorem \ref{Thm:FiniteModesError}.

To obtain the error due to the IP-DG method, we start with a lemma which 
is a direct result of Theorem \ref{Thm:MultiModesIP-DGError}.
\begin{lemma} \label{Lem:FinalIP-DGError}
Suppose that $\bfE_n \in \bL^2(\Om,\bH^2(D))$ for each $n \geq 0$ and $h$ is 
chosen to satisfy the following asymptotic mesh condition $k^3h^2 = O(1)$, then 
the following error estimates hold:
\begin{align*}
		\E \big( \| \bfE^\veps_{N} - \bfE^\veps_{h,N} \|_{L^2(D)} \big) &\leq 	
		\widetilde{C}_0(1+k) 
		h^2 \sum_{n = 0}^{N-1} \sum_{j = 0}^n \veps^n \big[ \widehat{C}_0 
		(2k + 2) \big]^{n-j} \E 
		\big( \| \bfS_j \|_{L^2(D)} \big), \\
		\E \big( \| \bfE^\veps_N - \bfE^\veps_{h,N} \|_{DG} \big) & \leq C 
		\widetilde{C}_0 (1+k) h \sum_{n = 0}^{N-1} \sum_{j = 0}^n \veps^n \big[ 
		\widehat{C}_0 (2k + 2) \big]^{n-j} \E \big( \| \bfS_j \|_{L^2(D)} 
		\big),
	\end{align*}
	where $C$, $\widetilde{C}_0$, and $\widehat{C}_0$ are constants 
	independent of $k$ and $h$.
\end{lemma}

To obtain the final estimate to characterize the IP-DG error we combine the 
previous lemma with the stability estimates in Theorem \ref{Thm:MultiModes}.
\begin{theorem} \label{Thm:FinalIP-DGError}
	Suppose that $\bfE_n \in \bL^2(\Om,\bH^2(D))$ for each $n \geq 0$ and $h$ 
	is chosen to satisfy the asymptotic mesh condition $k^3h^2 = 
	O(1)$.  Further assume that $\veps$ is chosen small enough to ensure $
	\widehat{\sigma} :=  14 \widehat{C}_0 \sqrt{C_0}(1 + k)(1+\mu)\veps < 1$. 
	Then the following estimates hold:
	\begin{align}
		\E \big( \| \bfE^\veps_N - \bfE^\veps_{h,N} \|_{L^2(D)} \big) &\leq
		C(C_0,\widehat{C}_0,\widetilde{C}_0,k,\veps)\,h^2\, \mathcal{M}
		(\bff)^{\frac{1}{2}}, 
		\label{Eq:FinalIP-DGError1}\\
		\E \big( \| \bfE^\veps_N - \bfE^\veps_{h,N} \|_{DG} \big) &\leq
		C(C_0,\widehat{C}_0,\widetilde{C}_0,k,\veps)\,h\,\mathcal{M}
		(\bff)^{\frac{1}{2}},
		\label{Eq:FinalIP-DGError2}
	\end{align}
	where
	\begin{align*}
		C(C_0,\widehat{C}_0,\widetilde{C}_0,k,\veps):=  
		\frac{C \widetilde{C}_0 \sqrt{C_0} (k + 1)}{7^{\frac{1}{2}}
		(7\sqrt{C_0}(1+\mu) - 1)} \cdot\frac{1}{1-\widehat{\sigma}}.
	\end{align*}
\end{theorem}
\begin{proof}
	To obtain \eqref{Eq:FinalIP-DGError1} and 
	\eqref{Eq:FinalIP-DGError2} we need to find an upper bound for the double 
	sum in the estimates in Lemma \ref{Lem:FinalIP-DGError}.  By using the 
	definition of $\bfS_j$ and Theorem \ref{Thm:MultiModes} we find
	\begin{align*}
		&\sum_{n=0}^{N-1} \sum_{j=0}^{n} \veps^n \big[ \widehat{C}_0 (2k+2) 
		\big]^{n-j} \E \big( \| \bfS_j \|_{L^2(D)} \big) \\
		& \qquad \leq \sum_{n=0}^{N-1} \sum_{j=0}^{n} \veps^n \big[ \widehat{C}_0 (2k+2) \big]^{n-j} \E \big( 2k^2 \| \bfE_{j-1} \|_{L^2(D)} + k^2 \| \bfE_{j-2} \|_{L^2(D)} \big) \\
		& \qquad \leq 4 (k+1) \mathcal{M}(\bff)^{\frac{1}{2}}\sum_{n=0}^{N-1} \sum_{j=0}^{n} \veps^n \big[ \widehat{C}_0 (2k+2) \big]^{n-j} C(j-1,k)^\frac{1}{2} \\
		& \qquad = \frac{4 \mathcal{M}(\bff)^{\frac{1}{2}}}{7^{\frac{3}{2}}(1+\mu)} \sum_{n=0}^{N-1} \big[\veps \widehat{C}_0(2k+2) \big]^n \sum_{j=0}^n 7^j C_0^{\frac{j}{2}}(1+\mu)^j \\
		& \qquad \leq \frac{4 \sqrt{C_0} \mathcal{M}(\bff)^{\frac{1}{2}}}{7^{\frac{1}{2}}\left(7 \sqrt{C_0}(1+\mu) - 1\right)} \sum_{n=0}^{N-1} \left[14 \veps \widehat{C}_0 \sqrt{C_0}(1+k)(1+\mu) \right]^n \\
		& \qquad =  \frac{4 \sqrt{C_0} \mathcal{M}(\bff)^{\frac{1}{2}}}{7^{\frac{1}{2}}\left(7 \sqrt{C_0}(1+\mu) - 1\right)} \cdot \frac{1- \left[14 \widehat{C}_0 \sqrt{C_0}(1+k)(1+\mu) \veps \right]^n}{1-14 \widehat{C}_0 \sqrt{C_0}(1+k)(1+\mu) \veps}.
	\end{align*}
	Combining the above estimate with the estimates in Lemma \ref{Lem:FinalIP-DGError} and using $\widehat{\sigma} < 1$ yield \eqref{Eq:FinalIP-DGError1} and \eqref{Eq:FinalIP-DGError2}.
\end{proof}

To characterize the sampling error, we combine Theorem \ref{Thm:MultiModesMCError} 
with a similar argument as used in Theorem \ref{Thm:FinalIP-DGError}.
\begin{theorem} \label{Thm:FinalMCError}
	Suppose that $h$ is chosen to satisfy the asymptotic mesh 
	condition $k^3h^2 = O(1)$ and $\veps$ is chosen small enough to satisfy $
	\widetilde{\sigma} := 4 \widehat{C}_0(1+k) \veps < 1$, then
	\begin{align}
		\E \big( \| \E(\bfE^\veps_{h,N}) - \bfPsi^\veps_{h,N} \|_{L^2(D)} 
		\big) & \leq \frac{\widehat{C}_0}{2 \sqrt{M}} \left( \frac{1}{k} + 
		\frac{1}{k^2} \right) \cdot \frac{1}{1-\widetilde{\sigma}} 
		\E \big( \| \bff \|_{L^2(D)} \big),
		\label{Eq:FinalMCError1} \\
		\E \big( \| \E(\bfE^\veps_{h,N}) - \bfPsi^\veps_{h,N} \|_{DG} 
		\big) & \leq \frac{\widehat{C}_0}{2 \sqrt{M}} \left( 1 + \frac{1}{k} 
		\right) \cdot \frac{1}{1-\widetilde{\sigma}} 
		\E \big( \| \bff \|_{L^2(D)} \big).
		\label{Eq:FinalMCError2}
	\end{align}
\end{theorem}
\begin{proof}
	From Theorem \ref{Thm:MultiModesMCError} the following estimates hold
	\begin{align*}
		&\E \big( \| \E( \bfE_{h,N}^\veps) - \bfpsi^\veps_{h,N} \|_{L^2(D)} 
		\big) \\
		& \qquad \leq  \sum_{n=0}^{N-1} \veps^n \E \big( \| \E( \bfE_{n}^h) - 
		\bfphi^h_{n} \|_{L^2(D)} \big) \\
		& \qquad \leq \frac{1}{\sqrt{M}} \left( \frac{1}{k} + 
		\frac{1}{k^2} \right) \E \big( \| \bff \|_{L^2(D)} \big)
		\sum_{n=0}^{N-1} \veps^n \widehat{C}(n,k)^{\frac{1}{2}} \\
		& \qquad = \frac{\widehat{C}_0}{2 \sqrt{M}} \left( \frac{1}{k} + 
		\frac{1}{k^2} \right) \E \big( \| \bff \|_{L^2(D)} \big)
		\sum_{n=0}^{N-1} \left[ 4 \widehat{C}_0 (k + 1) \veps \right]^n \\
		& \qquad = \frac{\widehat{C}_0}{2 \sqrt{M}} \left( \frac{1}{k} + 
		\frac{1}{k^2} \right) \cdot \frac{1 - \big(4 \widehat{C}_0(k+1) 
		\veps \big)^N}{1 - 4 \widehat{C}_0(k+1) \veps} 
		\E \big( \| \bff \|_{L^2(D)} \big).
	\end{align*}
	Thus, \eqref{Eq:FinalMCError1} holds.  \eqref{Eq:FinalIP-DGError2} is 
	proven using the same argument.
\end{proof}

Since each component of the error  
decomposition \eqref{Eq:Error_Decomposition} has been analyzed separately (see 
Theorems \ref{Thm:FiniteModesError}, \ref{Thm:FinalIP-DGError}, 
\ref{Thm:FinalMCError}), we now combine them to obtain the complete error analysis.
\begin{theorem} \label{Thm:Convergence}
Under the assumption that $\bfE_n \in \bL^2(\Om,\bH^2(D))$ for each $n \geq 0$, 
$h$ is chosen to satisfy the asymptotic mesh condition $k^3h^2 = O(1)
$, and $\veps$ is chosen small enough so that $\widehat{\sigma} < 1$, the 
following estimates hold:
\begin{align}
	\E \big( \| \E(\bfE^\veps) - \bfPsi^\veps_{h,N} \|_{L^2(D)} \big) & \leq C_1 
	\veps^N + C_2 h^2 + C_3 M^{\frac{1}{2}}, \label{Eq:Convergence1} \\
	\E \big( \| \E(\bfE^\veps) - \bfPsi^\veps_{h,N} \|_{L^2(D)} \big) & \leq C_1 
	\veps^N + C_2 h + C_3 M^{\frac{1}{2}}, \label{Eq:Convergence2}
\end{align}
where $C_j = C_j(C_0,\widehat{C}_0,\widetilde{C}_0,k,\veps,\bff)$ for $j = 
1,2,3$ are positive constants.
\end{theorem}


\section{Numerical experiments} \label{sec:Numerical_Experiments}
In this section we present numerical experiments
to verify the accuracy and efficiency of the proposed MCIP-DG with the multi-modes expansion.
As a benchmark, we will compare the multi-modes MCIP-DG approximation 
$\bfPsi_N = \bfPsi^\veps_{h,N}$ generated by Algorithm 2 with
the standard MCIP-DG method approximation $\bftPsi = \bftPsi^\veps_h$ generated 
by Algorithm 1.  Though theoretically the convergence theorem requires that 
$P \left\{ \om \in \Om; \| \nabla \eta(\om,\cdot) \|_{L^\infty(D) }\leq \mu \right\} = 1$ and 
the perturbation parameter  $\veps = O ((1+k)^{-1}(1 + \mu)^{-1})$, we will investigate both the smooth and non-smooth random fields.


In all of our numerical experiments the spatial domain $D$ is taken to be the unit cube 
$(0,1)^3$.  To define the IP-DG method this domain will be partitioned 
uniformly into cubes with dimension $h = 1/10$.  The size of $h$ is not taken 
to be smaller to limit the size of the linear system that is generated by the 
method.  In all of our tests we carry out a Monte Carlo procedure with $M = 
1000$ samples computed.  All computational tests are completed in MATLAB using 
the same Mac computer with a 2 GHz Intel Core i7  processor and 8 GB 1600 MHz 
DDR3 RAM.  The wave number parameter $k$ will be chosen as $k = 2$ so that our 
relatively coarse spatial mesh will be able to resolve the wave in a 
sufficient manner. We choose the following source function
\begin{align} \label{Eq:RHSFunction}
	\bff(\om,\bfx) = \Big[ \exp\big(\i k(1+\xi(\om,\bfx) )x \big), \exp\big(\i 
	k(1+\xi(\om,\bfx)) y \big), \exp\big(\i k(1+\xi(\om,\bfx)) z \big) 
	\Big]^T,
\end{align}
where $\xi(\om,\cdot)$ is a random variable satisfying $\| \xi(\om, \cdot) \|
_{L^\infty(D)} \leq 1$ almost surely.  


\subsection{Numerical experiments with smooth random field} \label{subsec:Num_1}

This subsection discusses numerical experiments that were carried out using smooth 
random field.  In particular, the random field $\eta$ in \eqref{Eq:PDE1} and $\xi$ in \eqref{Eq:RHSFunction} were 
taken to be Gaussian random fields with an exponential covariance function 
with correlation length $\ell = 0.5$ (c.f. \cite{Lord_Powell_Shardlow}), i.e. 
the following covariance function was used to generate the random coefficients 
in this subsection
\begin{align*}
	C(\bfx_1,\bfx_2) = \exp\left( - \frac{\| \bfx_1 - \bfx_2 \|_{2}}{0.5} 
	\right).
\end{align*}
For simplicity the coefficients were sampled for each cube in the partition of 
the spatial domain D. Figure \ref{fig:SmoothNoiseSamples} gives two samples of 
the coefficient functions used in this subsection.

\begin{figure}[htbp]
\centerline{
\includegraphics[width=2.2in,height=1.8in]{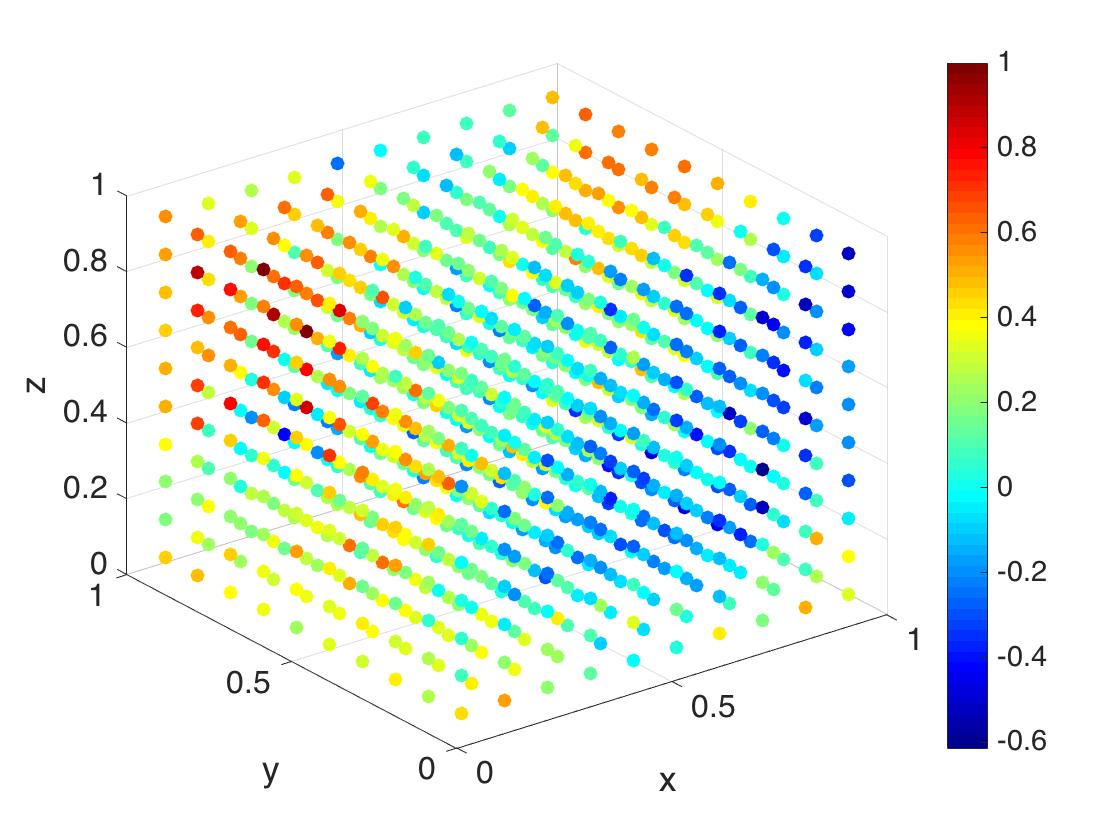} 
\includegraphics[width=2.2in,height=1.8in]{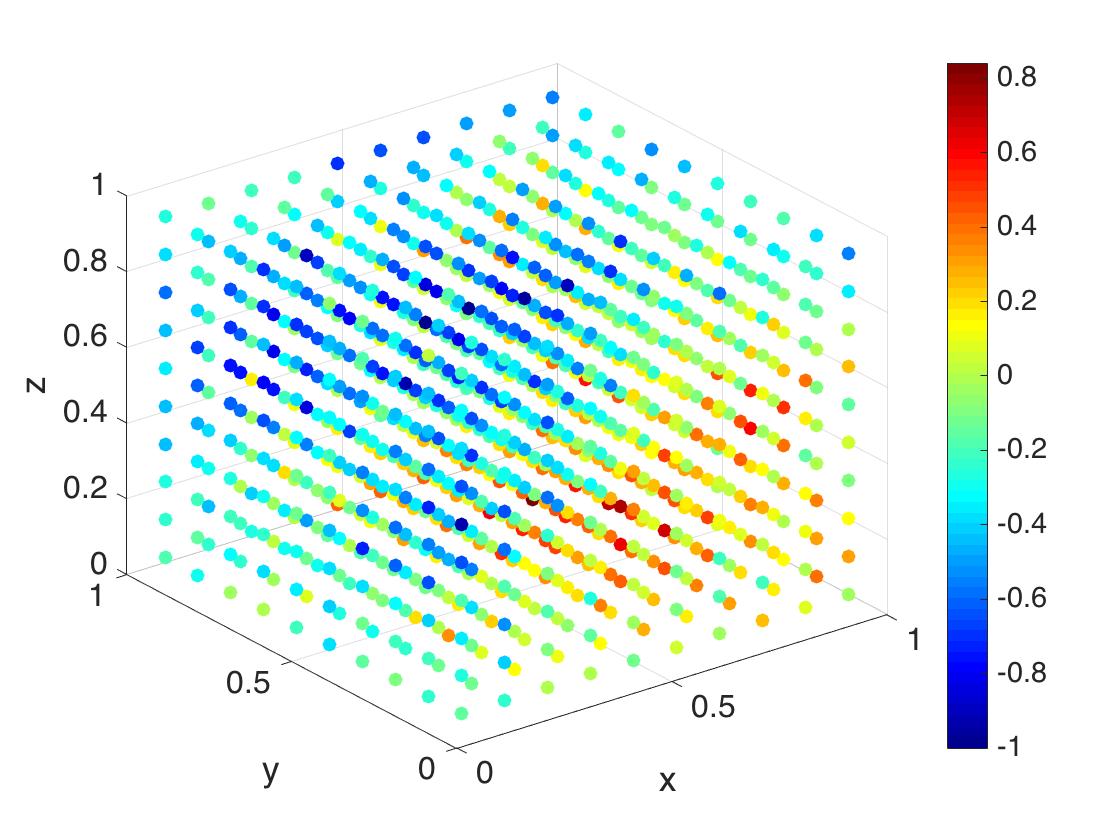}}
\centerline{
\includegraphics[width=2.2in,height=1.8in]{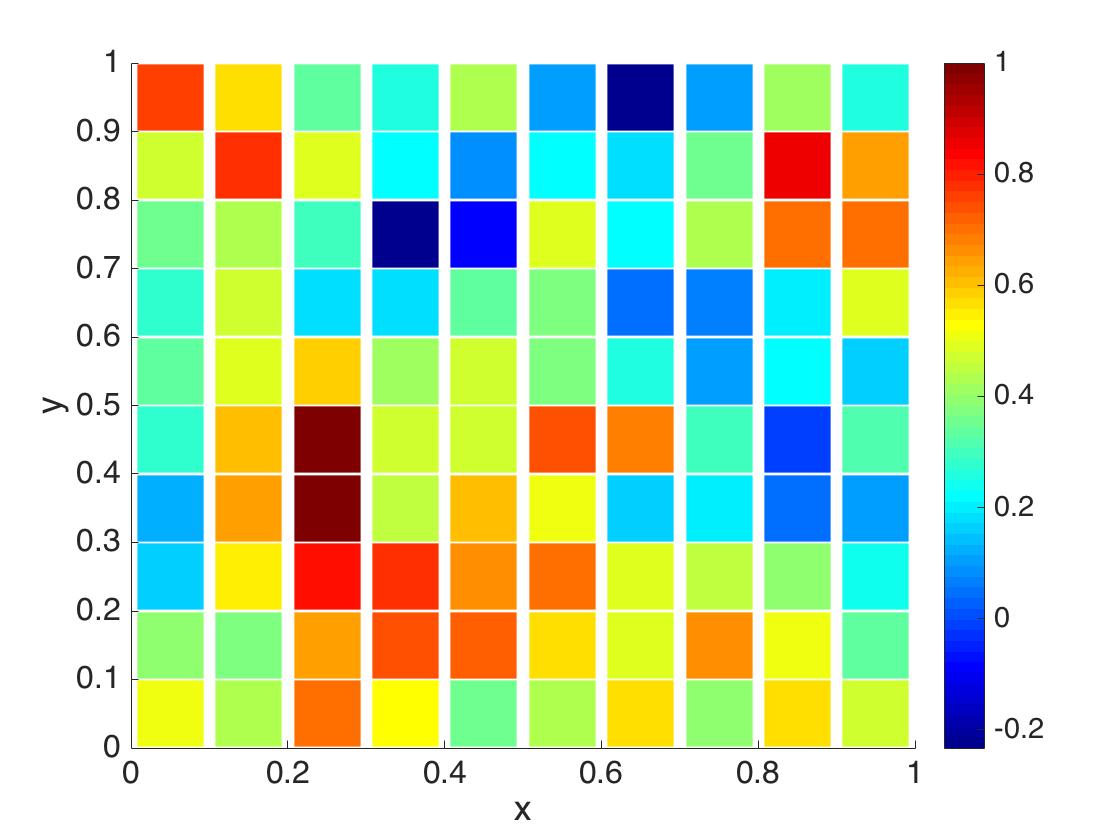} 
\includegraphics[width=2.2in,height=1.8in]{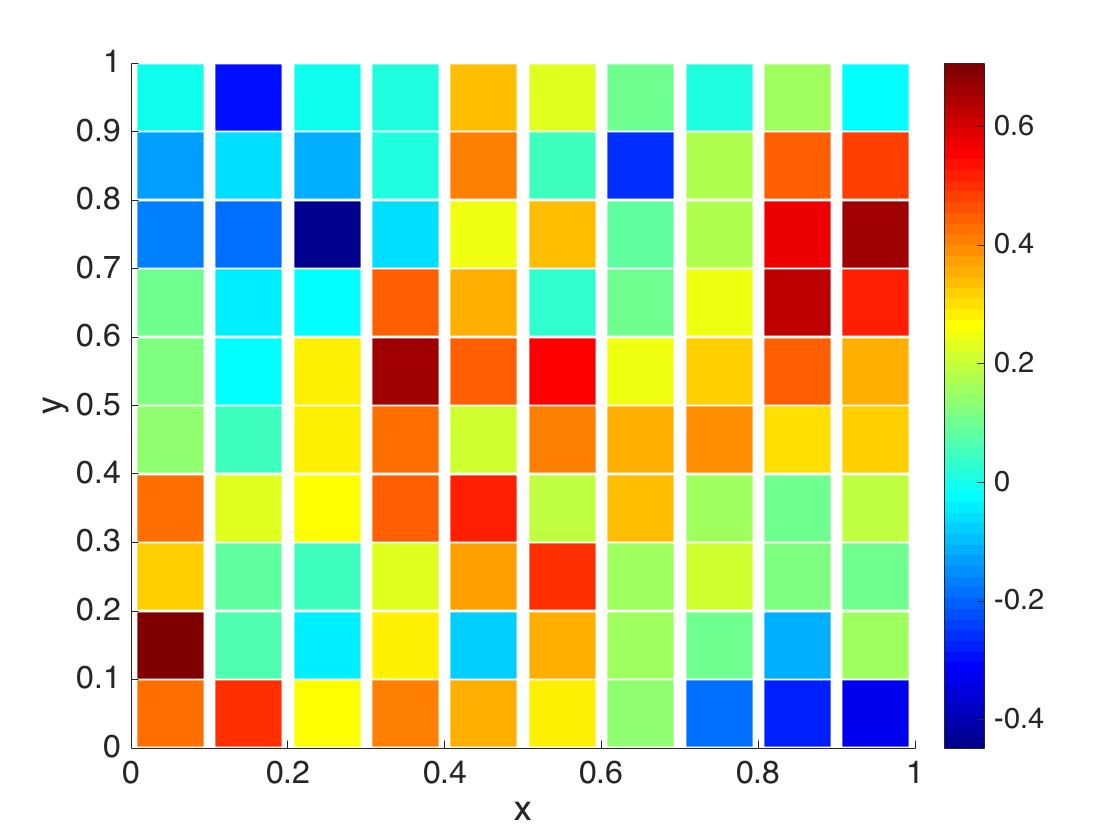}}
\caption{(Above)Samples of the random field $\eta(\omega,\cdot)$ generated 
using an exponential covariance function with covariance length $\ell = 0.5$ 
on a partition of $D$ parameterized by $h = 1/10$. 
(Below) Cross sections of these samples.}  \label{fig:SmoothNoiseSamples}
\end{figure}


We chose $\veps =$ 0.1, 0.3, 0.5, 0.7, 0.9, and for each fixed $\veps$, 
Algorithm 2 was used to produce the multi-modes MCIP-DG approximation $\bfPsi_N$ with $N =$ 0, 1, 2, 3, 4, 5, and 6.
Figure \ref{fig:Smooth_error_plot_1} -- \ref{fig:Smooth_error_plot_3} 
demonstrate the behavior of the error $\| \bftPsi - \bfPsi_N \|_{L^2(D)}$ and 
$\veps^N$ for these tests.  These plots use a log-scale for the y-axis for 
ease of comparison.  For all values of $\veps$ tested it is clear that the 
multi-modes MCIP-DG method produces an accurate approximation in comparison 
with the standard MCIP-DG method.  We also observe that the error converges at 
a rate $O(\veps^N)$ as predicted in the previous sections.  Surprisingly, we 
observe the method is working for a large $\veps$ value like $\veps = 0.9$.  
In previous tests involving the multi-modes MCIP-DG method applied to a random 
Helmholtz problem and random elastic Helmholtz problem the method stopped 
working for $\veps$ close to 1. See \cite{Feng_Lin_Lorton_15,Feng_Lorton_17}. 
This is possibly a result of the tests in this paper being carried out with a 
relatively small wave number parameter $k = 2$ and further tests should be 
carried out to investigate this.  

\begin{figure}[htbp]
\centerline{
\includegraphics[width=2.2in,height=1.8in]{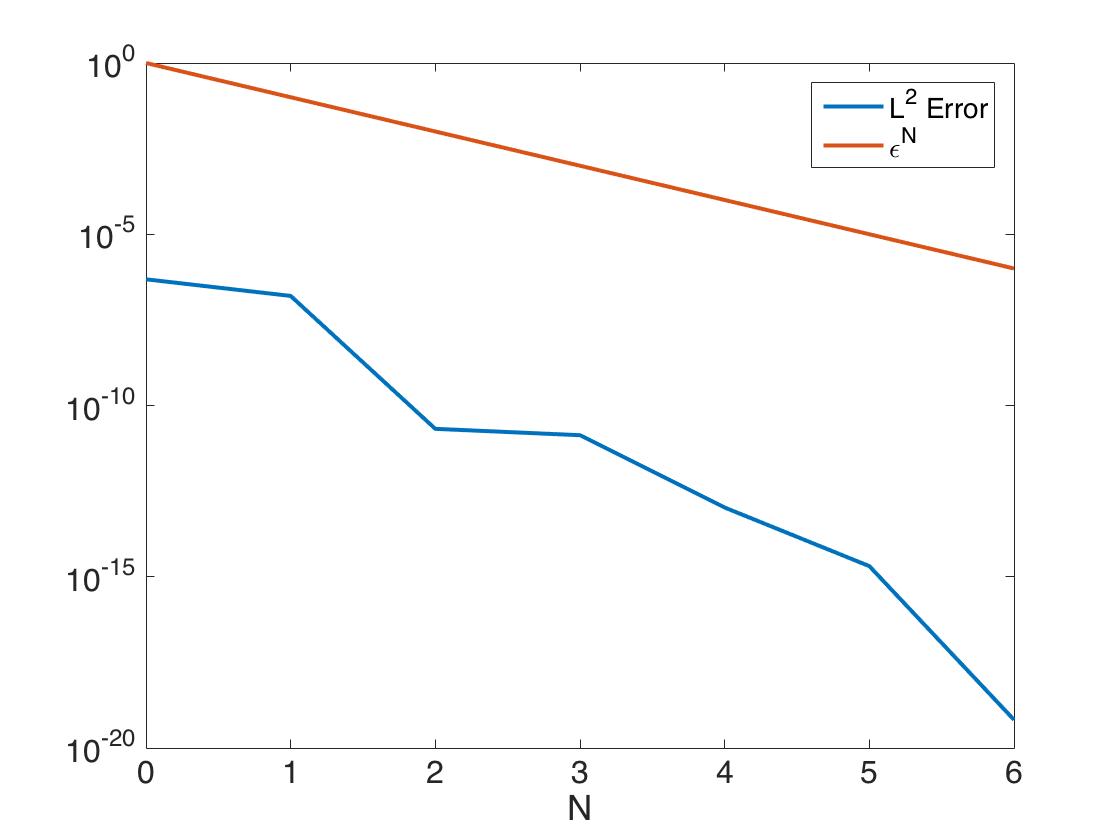} 
\includegraphics[width=2.2in,height=1.8in]{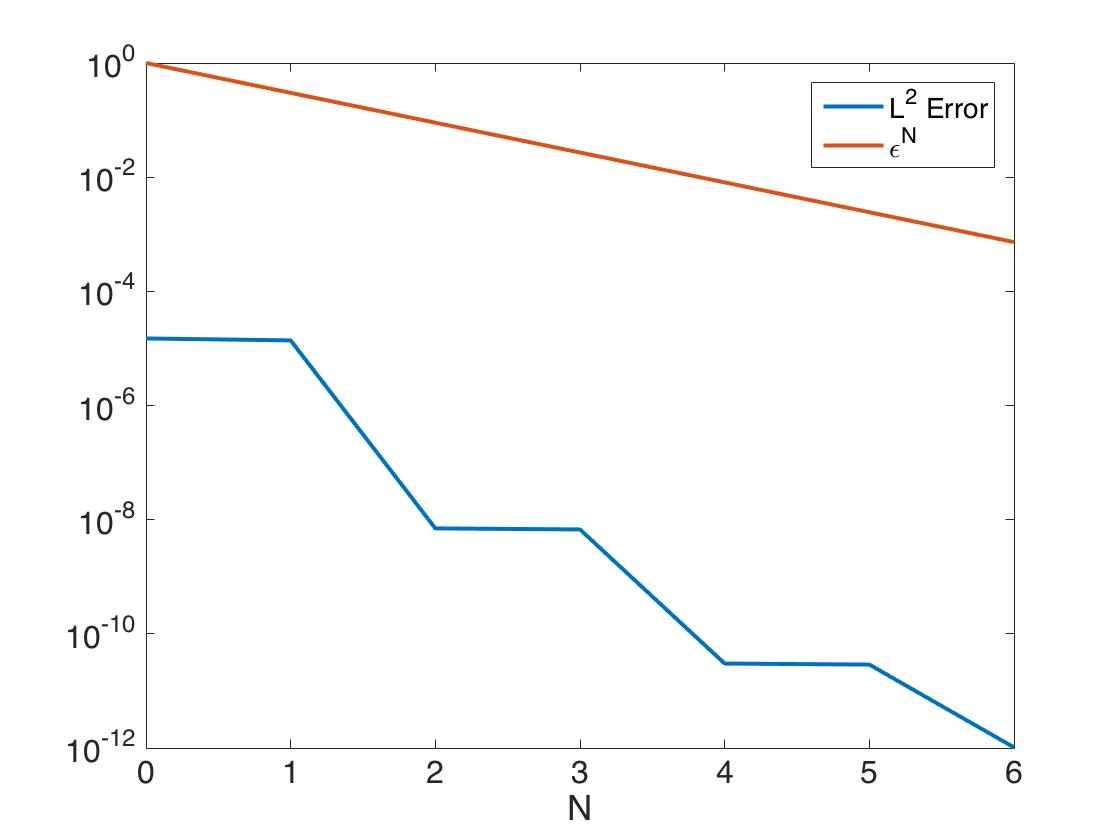}}
\caption{Plots of $\| \bftPsi - \bfPsi_N \|_{L^2(D)}$ and $\veps^N$ with $
\veps = 0.1$ (Left) and $\veps = 0.3$ (Right).}
\label{fig:Smooth_error_plot_1}
\end{figure}

\begin{figure}[htbp]
\centerline{
\includegraphics[width=2.2in,height=1.8in]{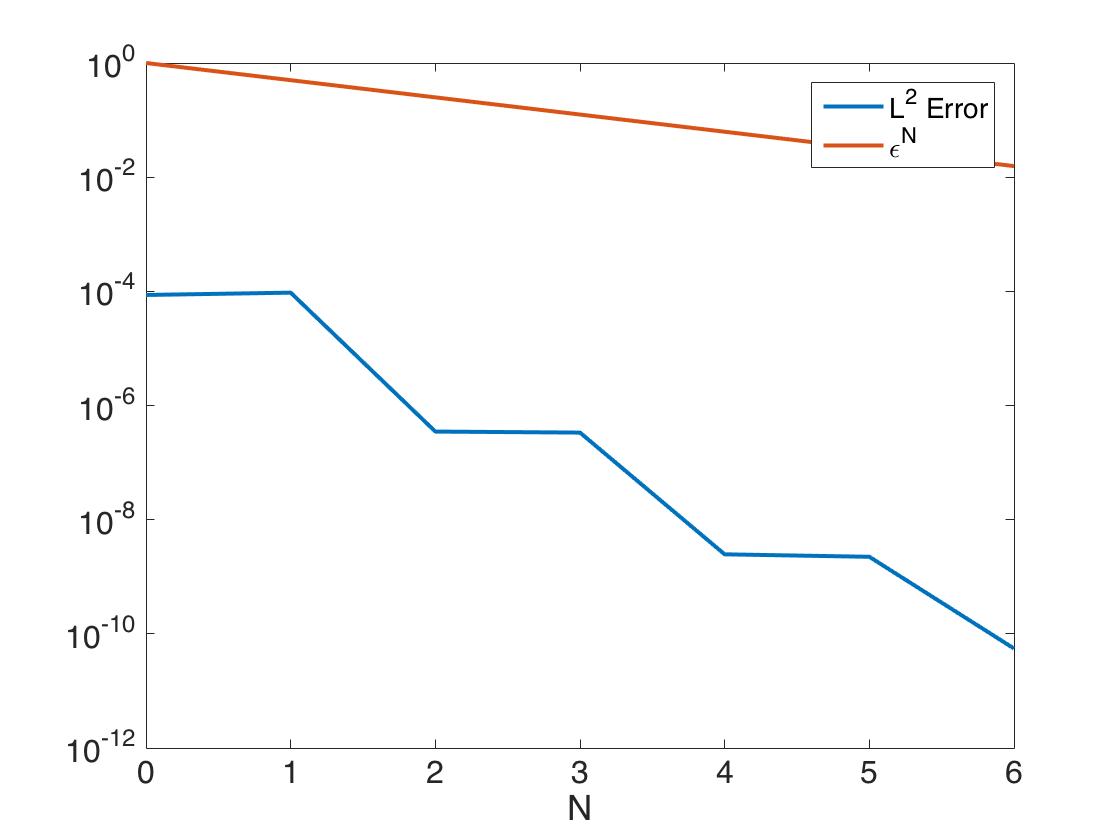} 
\includegraphics[width=2.2in,height=1.8in]{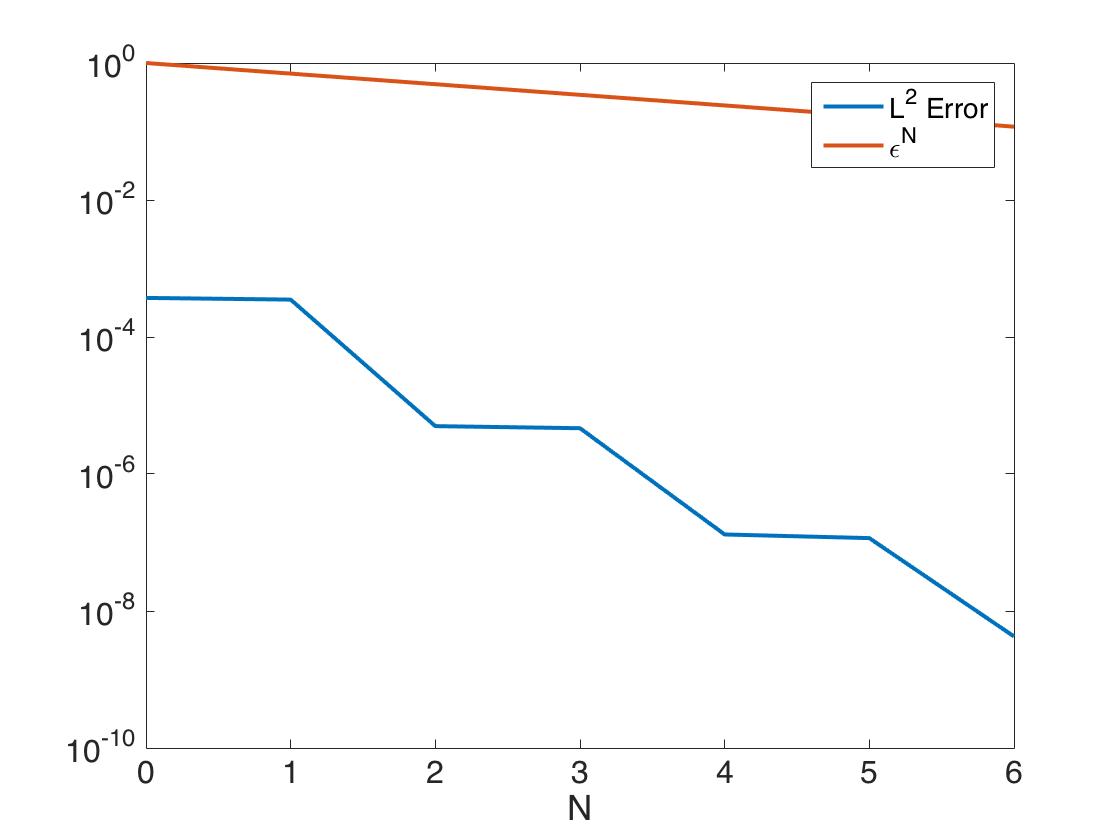}}
\caption{Plots of $\| \bftPsi - \bfPsi_N \|_{L^2(D)}$ and $\veps^N$ with $
\veps = 0.5$ (Left) and $\veps = 0.7$ (Right).}
\label{fig:Smooth_error_plot_2}
\end{figure}

\begin{figure}[htbp]
\centerline{
\includegraphics[width=2.2in,height=1.8in]{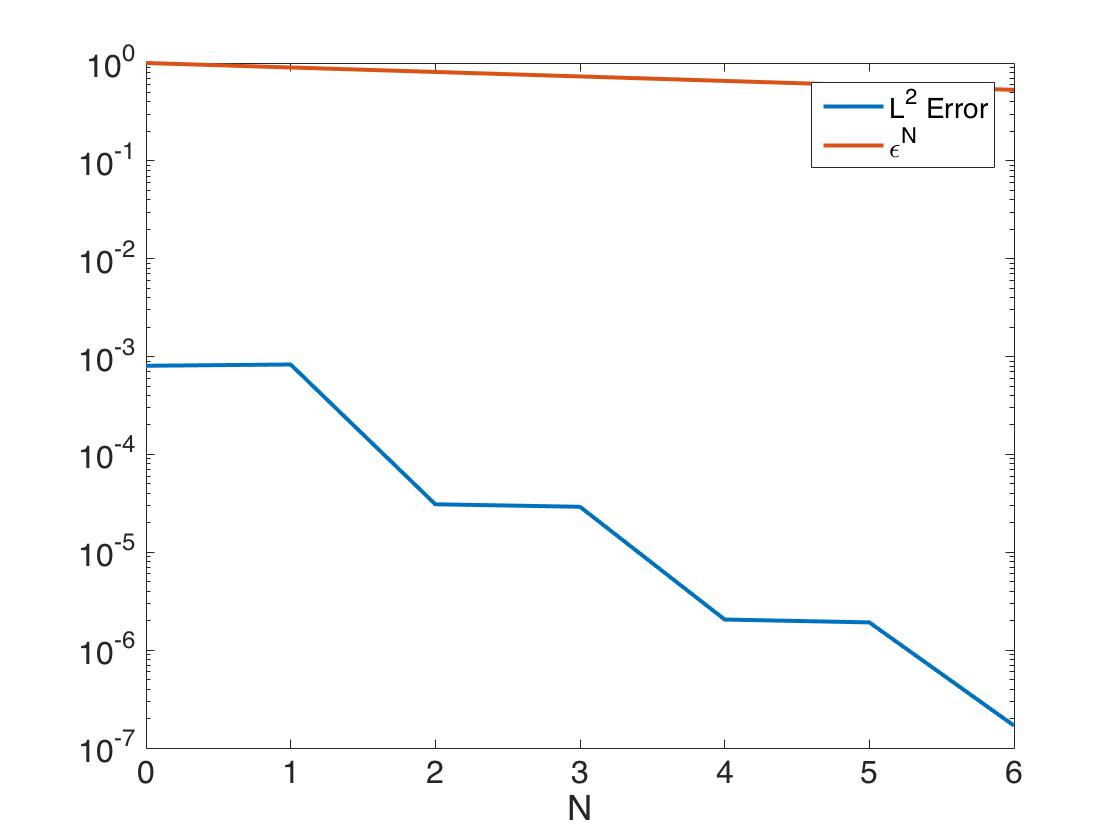} 
\includegraphics[width=2.2in,height=1.8in]{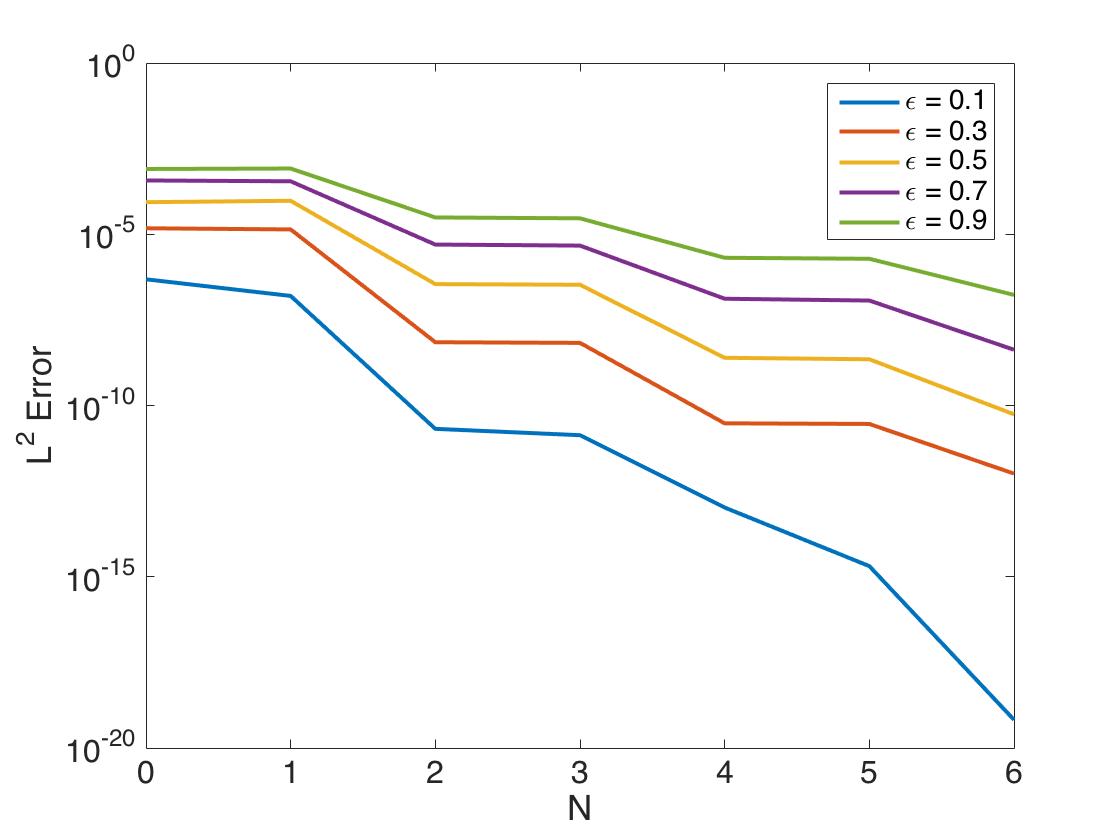}}
\caption{Plot of $\| \bftPsi - \bfPsi_N \|_{L^2(D)}$ and $\veps^N$ with 
$\veps = 0.9$ (Left).  Plots of $\| \bftPsi - \bfPsi_N \|_{L^2(D)}$ for 
varying values of $\veps$ (Right).}
\label{fig:Smooth_error_plot_3}
\end{figure}

We also observe that the error exhibits a behavior of staying relatively flat 
for approximations with $N$ odd while decreasing with $N$ even.  This behavior 
was also observed for other Helmholtz-like problems (c.f. 
\cite{Feng_Lin_Lorton_15,Feng_Lorton_17}).  This might lead one to believe 
that only even-labeled mode functions are useful in the multi-modes 
approximation, but this would be incorrect since the recursive relationship 
used to build the multi-modes approximation in \eqref{Eq:ModePDE3} involves 
both odd and even mode functions. From these results it does make more sense 
to apply the multi-modes MCIP-DG method with $N$ even as it is expected this 
will result in less error.

We also observe that for $\veps$ small $N$ can be chosen to be relatively 
small to obtain an accurate approximation.  This will lead to great savings in 
the computation time used to generate the approximation. Table 
\ref{table:ModesVsFullMonte} summarizes the computational time used for several modes.
As expected the multi-modes MCIP-DG approximation 
saves a great amount of time in comparison to the standard MCIP-DG 
approximation.  We also observe linear growth in computation time as the 
number of modes $N+1$ used to generate the approximation increase.

\begin{table}[htb]
\centering
\begin{tabular}{| c | c |}
	\hline
	Approximation & CPU Time (s) \\
	\hline
	$\bftPsi$ & $41832$ \\
	\hline
	$\bfPsi_0$ & $1436.8$ \\
	\hline
	$\bfPsi_1$ & $2738.1$ \\
	\hline
	$\bfPsi_2$ & $4041.4$ \\
	\hline
	$\bfPsi_3$ & $5343$ \\
	\hline
	$\bfPsi_4$ & $6647.7$ \\
	\hline
	$\bfPsi_5$ & $7953.8$ \\
	\hline
	$\bfPsi_6$ & $9261.9$ \\
	\hline
\end{tabular}
\caption{CPU times required to compute the multi-modes MCIP-DG approximation $
\bfPsi_N$ and standard MCIP-DG approximation $\bftPsi$.} 
\label{table:ModesVsFullMonte}
\end{table}

\subsection{Numerical experiments with non-smooth random field} \label{subsec:Num_2}

This subsection discusses numerical experiments that were carried out using non-smooth 
random coefficients.  In particular, the random coefficient $\eta$ in 
\eqref{Eq:PDE1} and the random parameter $\xi$ in \eqref{Eq:RHSFunction} were 
generated by sampling a uniformly distributed random variable for each cube 
in the partition of $D$ independently. Thus $\eta$ no longer satisfies the condition that it is smooth with $\| \nabla \eta \|_{L^\infty(D)} \leq \mu$ a.s. Figure \ref{fig:WhiteNoiseSamples} 
gives two samples of the coefficient functions used in this subsection.

\begin{figure}[htbp]
\centerline{
\includegraphics[width=2.2in,height=1.8in]{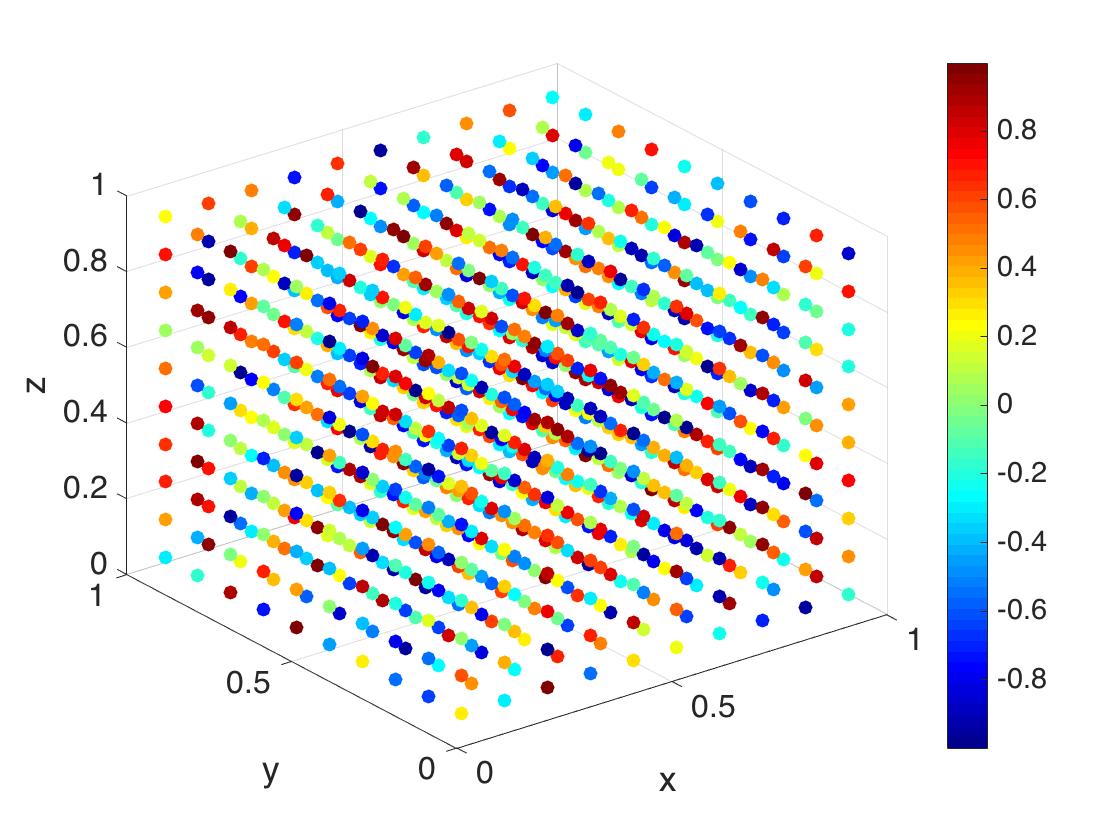} 
\includegraphics[width=2.2in,height=1.8in]{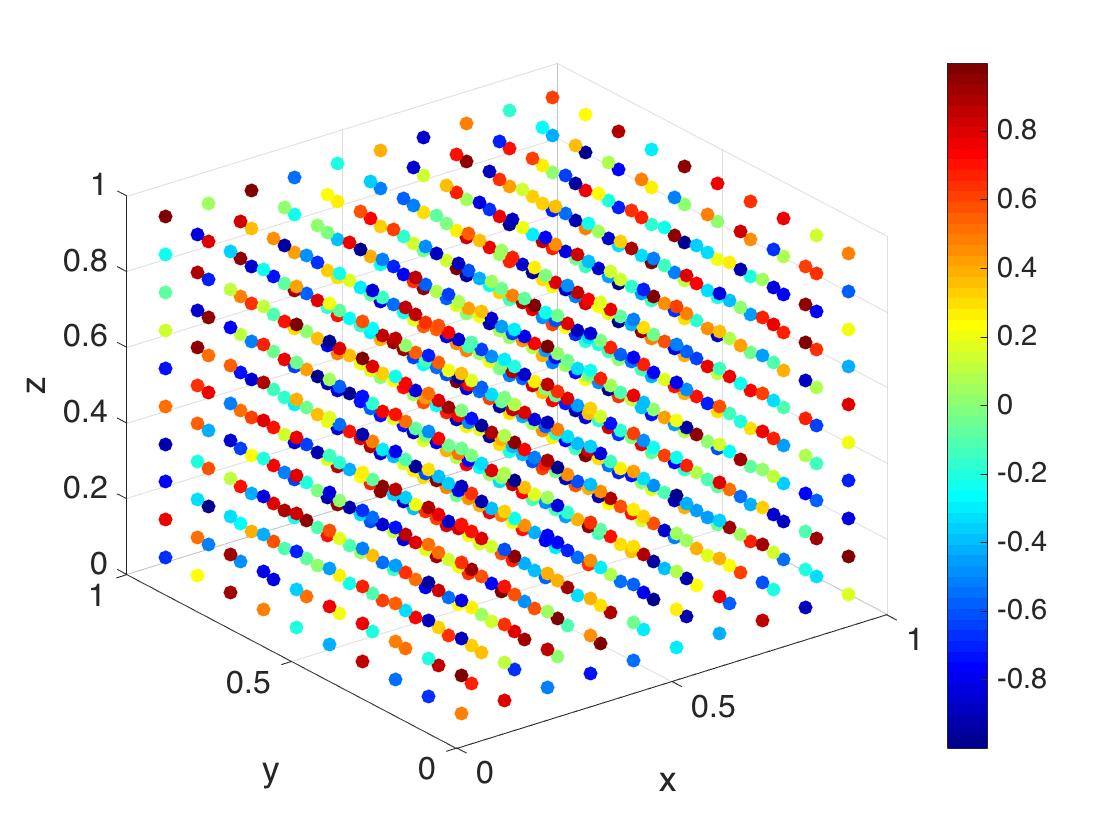}}
\centerline{
\includegraphics[width=2.2in,height=1.8in]{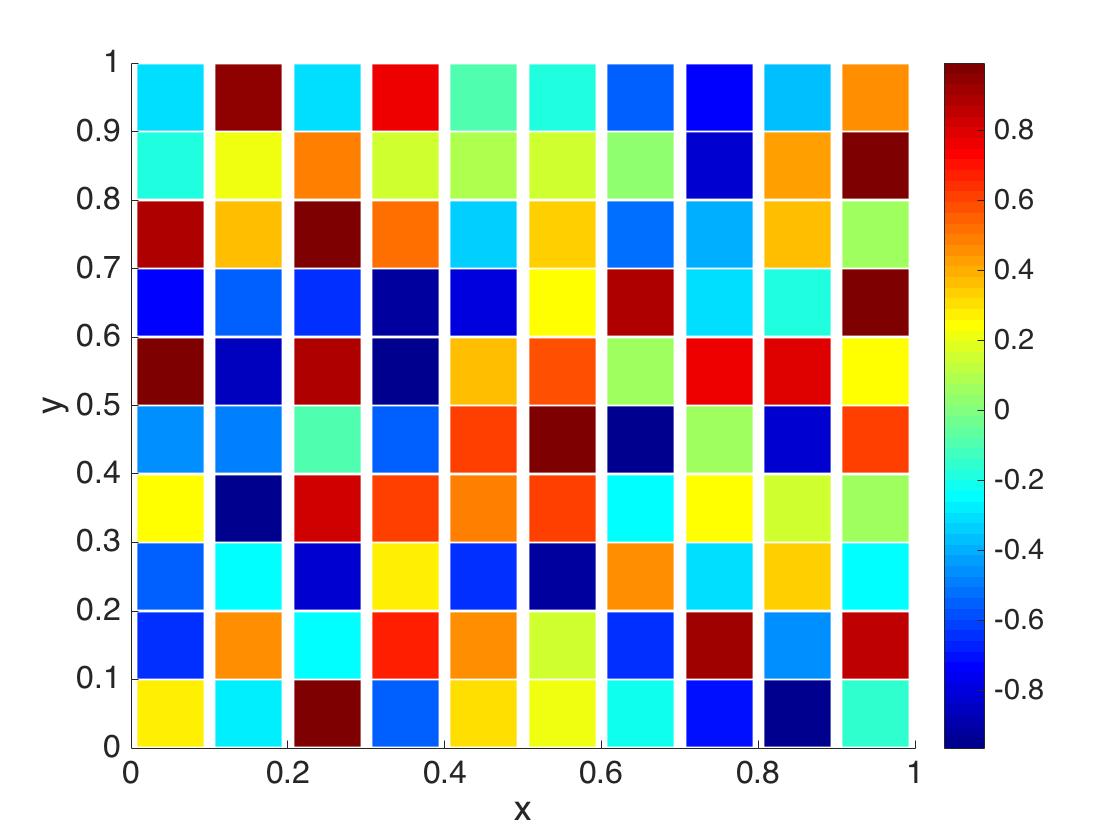} 
\includegraphics[width=2.2in,height=1.8in]{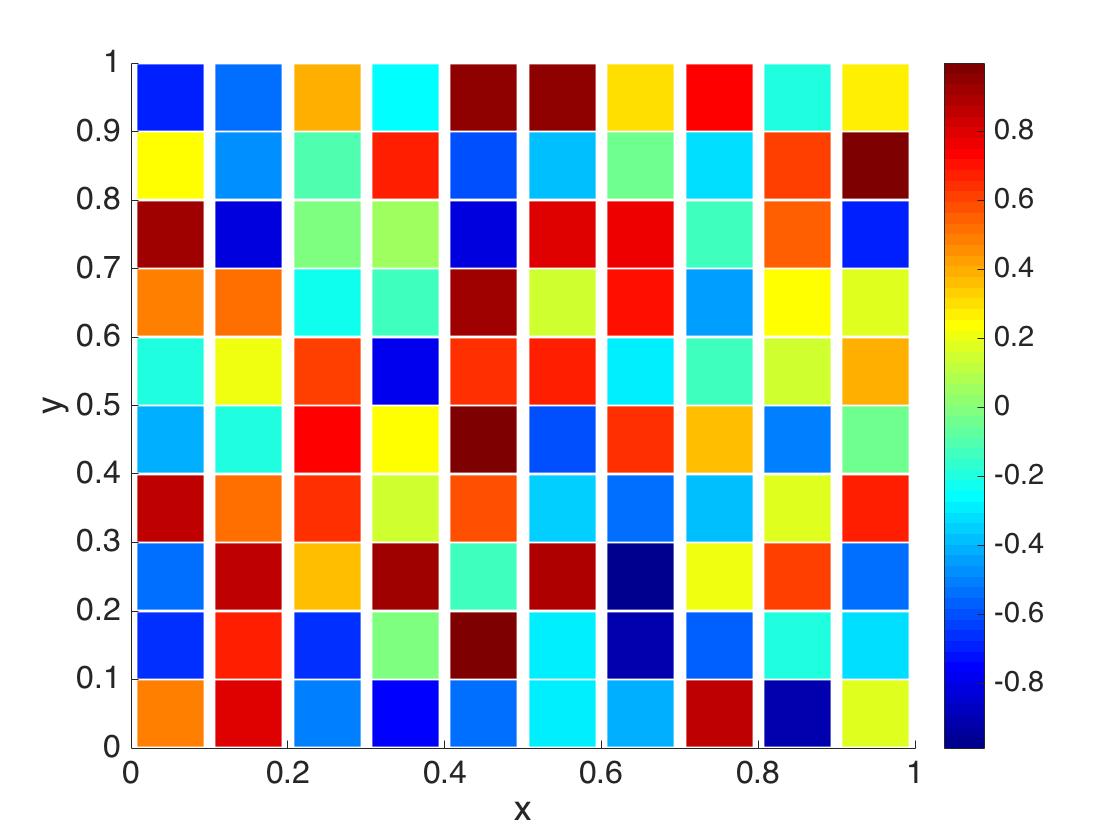}}
\caption{(Above) Samples of the random field $\eta(\omega,\cdot)$ generated 
using a uniformly distributed random variable for each cube in the partition 
of $D$ independently. (Below) Cross sections of these samples.} 
\label{fig:WhiteNoiseSamples}
\end{figure}

Experiments using non-smooth random coefficients $\eta$ and $\xi$ yielded similar results as demonstrated in Subsection \ref{subsec:Num_1}.  In particular, the method still demonstrated an error convergence rate $O(\veps^N)$ for $\veps$ = 0.1, 0.3, 0.5, 0.7, and 0.9. This is demonstrated
in Figure \ref{fig:Error_Plot_4}.
\begin{figure}[htbp]
\centerline{
\includegraphics[width=2.2in,height=1.8in]{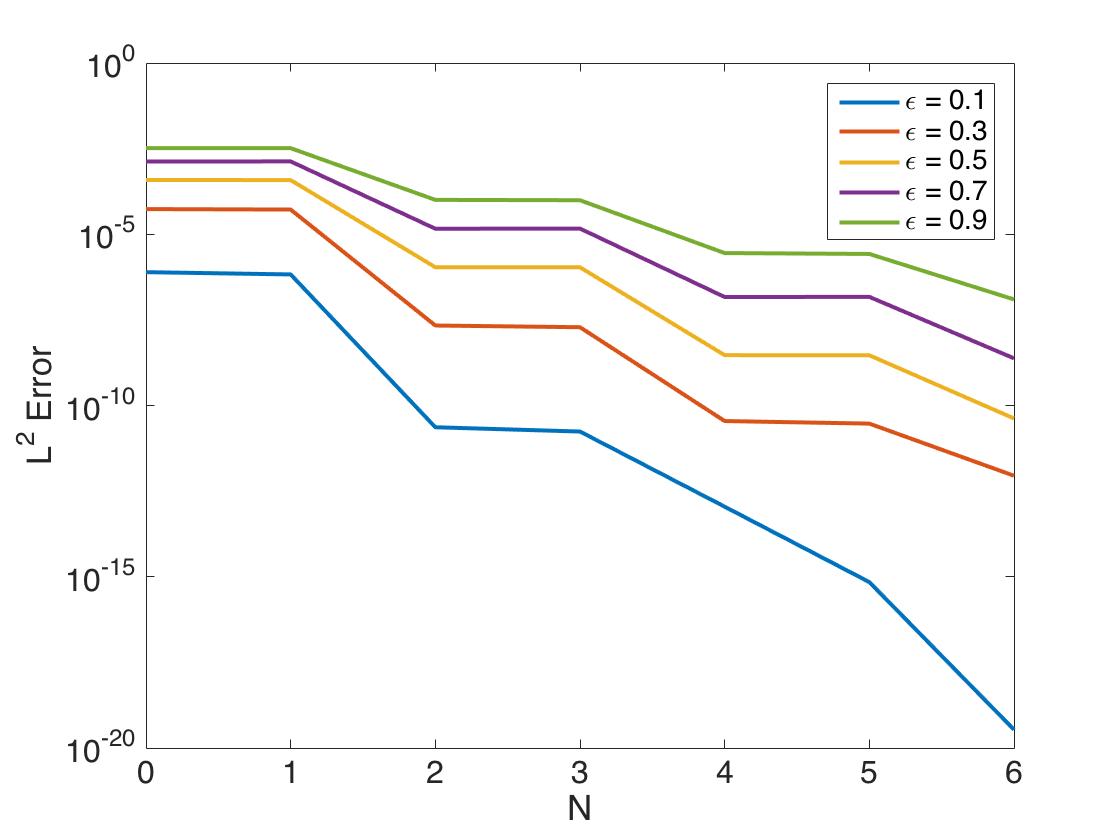}
} 
\caption{Plots of $\| \bftPsi - \bfPsi_N \|_{L^2(D)}$ for 
varying values of $\veps$.} \label{fig:Error_Plot_4}
\end{figure}

Due to the fact that the experiments with non-smooth random field returned similar convergence results ,we are hopeful that in some cases the added smoothness conditions on the random coefficient $\eta$ 
may be eliminated. More numerical experiments will be carried out to investigate.

\section{Extension to more general random media}\label{sec-7}
To use the multi-modes Monte Carlo DG method we have developed above, it requires that 
the random media are weak in the sense that the random coefficient $\alpha$ in the PDE has
the form $\alpha(\omega,\bfx)=\alpha_0 +\veps\eta(\omega,\bfx)$ and $\veps$ is not large (note that
we have taken $\alpha_0=1$ for notational brevity). In this section we present a procedure by which 
we can extend our multi-modes approach to a class of more general random media.

For general random media, the random coefficient $\alpha$ may not have 
the required ``weak form". To extend the multi-modes approach presented in the previous section 
to the general case,  our main idea is first to reformulate $\alpha(\omega,\bfx)$ into the 
desired ``weak form" $\alpha_0+\veps \eta(\omega,\bfx)$, then to apply the above ``weak" field framework.
There are at least two ways to do such a reformulation, the first one is to utilize the 
well-known Karhunen-Lo\`eve expansion and the second is to use  
a stochastic homogenization theory \cite{DGO}. Since the second approach is 
more involved and lengthy to describe, below we only outline the first approach.

For many biological and materials science applications, the random media can be described 
by a Gaussian random field \cite{FGPS, Ishimaru_78, Lord_Powell_Shardlow}. 
It is a well-known fact that any Gaussian random field $\alpha$ is uniquely determined by 
its mean and covariance function. Let $\alpha_0(\bfx)$ and $C(\bfx_1,\bfx_2)$ denote the mean and covariance 
function of the Gaussian random field $\alpha$, respectively. 
Two covariance functions, which are widely used in geoscience and materials
science, are $C(\bfx_1,\bfx_2)=\exp(-|\bfx_1-\bfx_2|^m/\ell)$ for $m=1,2$ and $0<\ell<1$ 
(cf. \cite[Chapter 7]{Lord_Powell_Shardlow}). Here $\ell$ is called correlation length which 
determines the range (or frequency) of the noise. We now recall that the Karhunen-Lo\`eve expansion 
for $\alpha(\ome,\bfx)$ takes the following form (cf. \cite{Lord_Powell_Shardlow}):
\[
\alpha(\omega,\bfx)= \alpha_0(\bfx) + \sum_{k=1}^\infty \sqrt{\lambda_k} \phi_k(\bfx) \xi_k(\omega),
\]
where $\{(\lambda_k, \phi_k)\}_{k\geq 1}$ is the eigenset of the (self-adjoint) covariance operator and
$\{\xi_k  \sim  N(0,1) \}_{k\geq 1}$ are i.i.d. random variables. It can be shown that  
$\lambda_k=O(\ell^r)$ for some $r>1$ depending on the spatial domain $D$ in which the PDE is defined 
(cf. \cite[Chapter 7]{Lord_Powell_Shardlow}). Consequently, for random media with small 
correlation length $\ell$, we have 
\[
\alpha(\omega,\bfx)=\alpha_0(\bfx)+ \sqrt{\lambda_1} \zeta(\omega,\bfx), \qquad 
\zeta(\omega,\bfx):= \sum_{k=1}^\infty\sqrt{ \frac{\lambda_k}{\lambda_1} }\, \phi_k(\bfx) \xi_k(\omega),
\]
Thus, setting $\varepsilon=\sqrt{\lambda_1} =O(\ell^{\frac{r}{2}})$ then leads to 
$\alpha(\ome,x)=\alpha_0 + \varepsilon \zeta$,
which has the desired ``weak form" which is given by a sum of a deterministic field and
a small random perturbation. As a result, our multi-modes Monte Carlo DG method 
is now applicable.  

We like to note that the classical Karhunen-Lo\`eve expansion may be replaced by other
types of expansion formulas which may result in more efficient multi-modes Monte Carlo methods. 
The feasibility and competitiveness of non-Karhunen-Lo\`eve expansion techniques will 
be investigated in a forthcoming work, where comparison among different expansion choices will 
also be studied. Finally, we remark that the DG method can be replaced by 
any other space discretization method such as finite difference, finite element, and
spectral methods in the main algorithm.


\def\cprime{$'$}

\end{document}